\documentclass[11pt,a4paper]{amsart}

\usepackage[left=3cm, top=3cm,bottom=3cm,right=3cm]{geometry}

\usepackage{amsthm,amsmath,amsfonts,amssymb}
\usepackage{mathdots}
\usepackage{graphicx}
\usepackage{bm}
\usepackage{mathtools}
\usepackage{type1cm}
\usepackage{latexsym}
\usepackage{mathrsfs}
\usepackage{dsfont}
\usepackage{bbm}
\usepackage{soul}
\usepackage{epic}
\usepackage{tikz}
\usetikzlibrary{automata, positioning}
\usepackage{enumitem}
\usepackage{url}

\newcommand{\R}{\mathbb{R}}
\newcommand{\E}{\mathbb{E}}
\newcommand{\N}{\mathbb{N}}

\newcommand{\Var}{\mathop{\mathrm{Var}}\nolimits}
\newcommand{\Cov}{\mathop{\mathrm{Cov}}\nolimits}

\renewcommand{\P}{\mathbb{P}}
\newcommand{\eps}{\varepsilon}

\newcommand{\ind}{\mathbbm{1}}
\newcommand{\1}{\mathbbm{1}}

\DeclareMathOperator*{\argmax}{arg\,max}
\DeclareMathOperator*{\argmin}{arg\,min}

\newcommand{\od}{\stackrel{d}{=}}

\newcommand{\todistr}{\overset{d}{\underset{n\to\infty}\longrightarrow}}

\newcommand{\toprobab}{\overset{P}{\underset{n\to\infty}\longrightarrow}}

\newcommand{\minorant}{S^\smile}
\newcommand{\majorant}{S^\frown}
\newcommand{\minlength}{L^\smile}
\newcommand{\majlength}{L^\frown}
\newcommand{\tminlength}{\tilde{L}^\smile}

\theoremstyle{plain}
\newtheorem{theorem}{Theorem}[section]
\newtheorem{lemma}[theorem]{Lemma}
\newtheorem{corollary}[theorem]{Corollary}
\newtheorem{proposition}[theorem]{Proposition}

\theoremstyle{definition}

\theoremstyle{remark}
\newtheorem{remark}[theorem]{Remark}

\begin{document}

\title{How long is the convex minorant of a one-dimensional random walk?}

\author[G.~Alsmeyer, Z.~Kabluchko, A.~Marynych and V.~Vysotsky]{Gerold Alsmeyer, Zakhar Kabluchko, Alexander Marynych\\ and Vladislav Vysotsky}

\begin{abstract}
We prove distributional limit theorems for the length of the largest convex minorant of a one-dimensional random walk with independent identically distributed increments.  Depending on the increment law, there are several regimes with different limit distributions for this length. Among other tools, a representation of the convex minorant of a random walk in terms of uniform random permutations is utilized.
\end{abstract}

\keywords{convex minorant, random permutation, random walk} \subjclass[2010]{Primary:
60F05, 60G55; Secondary: 60J10.}

\maketitle

\section{Introduction and main results}

Given a sequence $(\xi_{k})_{k\in\N}$ of independent and identically distributed (i.i.d.) real-valued random variables with a generic copy $\xi$, consider the associated random walk $(S_{n})_{n\in\N_{0}}$, $\N_{0}:=\N\cup\{0\}$, defined by
$$
S_{0}:=0\quad\text{and}\quad S_{n}:=\xi_{1}+\xi_{2}+\cdots+\xi_{n}\quad\text{for }n\in\N,
$$
and the random piecewise linear function $t\mapsto S(t)$, $t\ge 0$, obtained by linear interpolation between the values $S(n):=S_{n}$ for $n \in \N_{0}$. For any fixed $T>0$, let $t\mapsto \minorant_T(t)$ and $t\mapsto \majorant_T(t)$ be, respectively, the convex minorant and the concave majorant of the function $t\mapsto S(t)$ on the interval $[0,T]$. Let us recall that the convex minorant (concave majorant) of a function $f$ on an interval $[a,b]$ is the largest convex (least concave) function $f^{{}^\smile}$ ($f^{{}^\frown}$) such that $f^{{}^\smile}(x)\le f(x)$ ($f(x)\le f^{{}^\frown}(x)$) for all $x\in[a,b]$. Clearly, both  $t\mapsto \minorant_T(t)$ and $t\mapsto  \majorant_T(t)$ are piecewise linear continuous functions and have therefore well-defined finite lengths, here denoted by $\minlength_T$ and  $\majlength_T$, respectively.

\vspace{.1cm}
In this paper, we provide distributional limit theorems for $\minlength_{n}$ and $\majlength_{n}$ as $n\to\infty$ in the following three regimes:
\begin{itemize}\itemsep2pt
\item[(A)] $\E \xi^{2}<\infty$ and $\E \xi=0$; 
\item[(B)] the law of $\xi$ lies in the domain of attraction of an $\alpha$-stable law with $\alpha\in(1,2)$ and $\E \xi=0$;  
\item[(C)] the law of $\xi$ lies in the domain of attraction of an $\alpha$-stable law with $\alpha\in(0,1)$.
\end{itemize}
The case $\E \xi\neq 0$ in parts (A) and (B) turns out to be less intriguing and will be discussed in Section~\ref{sec:non_zero_mu}.

\vspace{.1cm}
Let us point out at the outset that it suffices to consider the length of the convex minorant $\minlength_{n}$ because it has the same law as $\majlength_{n}$, so
\begin{equation}\label{eq:minlength=majlength in law}
\minlength_{n}\,\od\,\majlength_{n}\quad\text{for all }n\in\N.
\end{equation}
Although nonintuitive, this follows fairly easily from the observation that the concave majorant of $(S_{0},\ldots,S_{n})$ for any $n$ coincides with the negative of the convex minorant of the reflected vector $(-S_{0},\ldots,-S_{n})$ in combination with a distributional representation of $\minlength_{n}$, stated as \eqref{eq:basic_representation2} and owing to Abramson et al. \cite{Abramson+Pitman:2011,Abramson+Pitman+Ross+Bravo:2011}, which only involves the squares of the $S_{k}$.

\vspace{.1cm}
Before putting our work into some context by pointing out connections with earlier work on convex minorants and the convex hulls of random walks, we present our main results, stated as Theorems \ref{thm:finite_variance}, \ref{thm:infinite_variance} and \ref{thm:infinite_mean}.

\vspace{.1cm}
It is well-known that in each of the cases (A), (B) and (C), there exists a sequence $(a_{n})_{n\in\N}$ of positive constants such that, with $\mathcal{S}_{\alpha}=\left(\mathcal{S}_{\alpha}(t)\right)_{t\in [0,1]}$ denoting an $\alpha$-stable L\'{e}vy process,
\begin{equation}\label{eq:clt_for_rw}
\left(\frac{S(nt)}{a_{n}}\right)_{t\in [0,1]}\ \xRightarrow{n\to\infty}\ (\mathcal{S}_{\alpha}(t))_{t\in[0,1]}
\end{equation}
in the Skorokhod space $D[0,1]$ endowed with the standard $J_{1}$-topology. Note that $\mathcal{S}_{2}$ is just a centered Brownian motion. Throughout the paper, we always use $a_{n}$ and $\mathcal{S}_{\alpha}$ for the normalization and the $\alpha$-stable L\'{e}vy process such that \eqref{eq:clt_for_rw} holds. Also, we let $\mathcal{N}(0, s^{2})$ denote the  normal distribution with mean zero and positive variance $s^{2}$.
The notation $\simeq$ stands for asymptotic equivalence, that is, $f(x)\simeq g(x)$, as $x\to x_0$, holds if and only if $\lim_{x\to x_0}f(x)/g(x)=1$.

\vspace{.1cm}
In order to state our result for case (A), put
\begin{equation}\label{eq:sigma_{n}_definition}
\sigma^{2}_{n}\ :=\Var\left(\xi \ind_{\{|\xi |\le \sqrt{n}\}}\right), \quad n \in \N,
\end{equation}
and $\log^{+}x:=\max(\log x,0)$ for $x>0$.

\begin{theorem}\label{thm:finite_variance}
Suppose that $\E \xi=0$ and $\E\xi^{2}=:\sigma^{2}=\lim_{n\to\infty}\sigma_{n}^{2}<\infty$.  
Then
\begin{equation}\label{eq:clt_finite_variance}
\frac{1}{\sqrt{\log n}}\left(\minlength_{n} - n - \sum_{j=1}^{n}\frac{\sigma_{j}^{2}}{2j}\right)\ \todistr\  \mathcal{N}\left(0,\frac{3\sigma^{4}}{4}\right),
\end{equation}
which may be simplified to
\begin{equation}\label{eq:clt_finite_variance2}
\frac{1}{\sqrt{\log n}}\left(\minlength_{n} - n - \frac{\sigma^{2}}{2}\log n\right)\ \todistr\ \mathcal{N}\left(0,\frac{3\sigma^{4}}{4}\right)
\end{equation}
under the additional assumption $\E\xi^{2}\log^{+} |\xi|<\infty$.
\end{theorem}

In view of the previous result, one could expect that in case (B) a suitable normalization of $\minlength_{n}$ converges in law to some stable law. It may therefore be surprising that the true answer, stated in the next theorem, looks more complicated.

\begin{theorem}\label{thm:infinite_variance}
Suppose that the following assumptions hold:
\begin{itemize}[leftmargin=1.1cm]\itemsep2pt
\item[(B1)] The function $t\mapsto \P\{|\xi|>t\}$ is regularly varying at infinity with index $\alpha \in (1,2)$ and $\E \xi=0$.
\item[(B2)] For $p,q\in[0,1]$ such that $p+q=1$, we have
$$
\lim_{x\to+\infty}\frac{\P\{\xi>x\}}{\P\{|\xi|>x\}}\,=\,p\quad\text{and}\quad\lim_{x\to+\infty}\frac{\P\{\xi<-x\}}{\P\{|\xi|>x\}}\,=\,q.
$$
\end{itemize}
Then
$$
\frac{n}{a_{n}^{2}}\left(\minlength_{n}-n\right)\ \todistr\ \frac{1}{2}\sum_{k=1}^{\infty}\frac{(\mathcal{S}_{\alpha}^{(k)}(Z_{k}))^{2}}{Z_{k}},
$$
where $(Z_{1},Z_{2},\ldots)$ is a random sequence that has a Poisson--Dirichlet distribution with parameter $\theta=1$ and  $(\mathcal{S}_{\alpha}^{(k)}(t))_{t\in[0,1]}$, $k=1,2, \ldots,$ are independent copies of the $\alpha$-stable process $(\mathcal{S}_{\alpha}(t))_{t\in[0,1]}$ appearing in \eqref{eq:clt_for_rw}.
\end{theorem}

The proofs of both theorems rely on a known representation of the convex minorant in terms of uniform random permutations that will be described in Subsection \ref{subsec:uniform_perm}, followed by some explanations of the main arguments in Subsection~\ref{subsec:explanation}. The main difference between the cases (A) and (B) is that, roughly speaking, in case (A) the main contributions to the fluctuations of $\minlength_{n}$ are due to a large number of ``small'' segments of the convex minorant, whereas in case (B) they are rather due to few ``large'' segments.

\vspace{.1cm}
Our third result deals with the case when $\xi$ lies in the domain of attraction of a stable law with index $\alpha\in(0,1)$. This is the simplest case because rather than making use of
the connection with random permutations, a simple comparison argument applies; see Subsection \ref{sec:proof_inf_mean} below. 

\begin{theorem}\label{thm:infinite_mean}
Suppose that the following assumptions hold:
\begin{itemize}[leftmargin=1.1cm]\itemsep2pt
\item[(C1)] The function $t\mapsto \P\{|\xi|>t\}$ is regularly varying at infinity with index $\alpha \in (0,1)$.
\item[(C2)] For some $p,q\in[0,1]$ with $p+q=1$,
$$
\lim_{x\to+\infty}\frac{\P\{\xi>x\}}{\P\{|\xi|>x\}}\,=\,p\quad\text{and}\quad\lim_{x\to+\infty}\frac{\P\{\xi<-x\}}{\P\{|\xi|>x\}}\,=\,q.
$$
\end{itemize}
Then
$$
\left(\frac{\minlength_{n}}{a_{n}}, \frac{\majlength_{n}}{a_{n}}\right)\todistr \left(\mathcal{S}_{\alpha}(1)-2\inf_{t\in [0,1]}\mathcal{S}_{\alpha}(t), 2\sup_{t\in [0,1]}\mathcal{S}_{\alpha}(t)-\mathcal{S}_{\alpha}(1)\right),
$$
where $(a_{n})_{n\in\N}$ and $(\mathcal{S}_{\alpha}(t))_{t\in[0,1]}$ are as in \eqref{eq:clt_for_rw}.
\end{theorem}

\begin{remark} \label{rem:non-joint}\rm
There is an interesting connection of our results, notably Theorem \ref{thm:finite_variance}, with the work by Wade and coauthors \cite{McRedmond+Wade:2018,Wade+Xu:2015a,Wade+Xu:2015b} on the convex hulls of planar random walks. Assuming $\E\xi =0$, one can regard the bivariate sequence $\{(n,S_{n})\}_{n\in\N_{0}}$ as a degenerate walk in the plane whose increments are supported on the line orthogonal to the mean vector $(1,0)$ and thus to the $x$-axis. Except for this degenerate case, it was shown by Wade and Xu \cite[Thms. 1.1 and 1.2]{Wade+Xu:2015b} that the perimeter of the convex hull of the first $n$ steps of any square-integrable planar random walk satisfies a central limit theorem and has linearly growing variance as $n \to \infty$. But in the degenerate case, their approach only provides that this growth is sublinear. Note that Theorem~\ref{thm:finite_variance} provides no information on the asymptotic behavior of the moments of $\minlength_{n}$, as it does not claim any type of uniform integrability. The moment asymptotics are specified in Theorem~\ref{thm:moments} below. Remarkably, the variance of $\minlength_{n}$ grows logarithmically if $\xi$ has a finite third moment but it may grow polynomially when $\E |\xi|^3 = \infty$, see \eqref{eq:prop_variance_large} and~\eqref{eq:prop_variance_log_asymp}.

\vspace{.1cm}
Note further that the perimeter $L_{n}$ of the convex hull of $\{(j,S_{j}): j=0,\ldots,n\}$ equals $\minlength_{n}+\majlength_{n}$. Under the assumptions of our Theorem \ref{thm:infinite_mean} (Case (C)), we therefore immediately infer a distributional limit result for $L_{n}$, namely 
$$
\frac{L_{n}}{2a_{n}}\ \todistr\  \sup_{t\in [0,1]}\mathcal{S}_{\alpha}(t)-\inf_{t\in [0,1]}\mathcal{S}_{\alpha}(t).
$$
On the other hand and despite relation \eqref{eq:minlength=majlength in law}, we do not know in the other cases whether \emph{joint} convergence of $(\minlength_{n},\majlength_{n})$ holds which would give a limit theorem for the perimeter $L_{n}$ in all cases. The connection with random permutations seems to be insufficient for this purpose and we  leave this as an open problem. 
\end{remark}


\begin{theorem}\label{thm:moments}
Suppose that $\E\xi=0$ and $\sigma^{2}=\E\xi^{2} \in (0,\infty)$. Then
\begin{equation}\label{eq:prop_first_moment}
\E \minlength_{n} -n\ = \ \frac{\sigma^{2}}{2}\log n + o(\log n) \quad\text{as }n\to\infty,
\end{equation}
where the term $o(\log n)$ can be replaced by $O(1)$ if and only if $\E\xi^{2}\log^{+} |\xi|<\infty$. 

\vspace{.1cm}\noindent
Furthermore, if $\E |\xi|^{p} < \infty$ for some $p \in [2, 3)$, then 
\begin{equation}\label{eq:prop_variance_poly_asymp}
\Var \minlength_{n}\ =\ o(n^{3-p})  \quad\text{as }n\to\infty
\end{equation}
and also
\begin{equation}\label{eq:prop_variance_large}
\Var \minlength_{n}\ \ge\ 0.02 n^3 \P (|\xi| \ge 2n) + O(1) \quad\text{as }n\to\infty.
\end{equation}

\vspace{.1cm}\noindent
Finally, if $\E |\xi|^3<\infty$, then
\begin{equation}\label{eq:prop_variance_log_asymp}
\Var \minlength_{n}\ \simeq\ \frac{3\sigma^{4}}{4}\log n\quad\text{as }n\to\infty.
\end{equation}
\end{theorem}
\begin{corollary} \label{cor:variance}
If $\E\xi=0$ and $\E|\xi|^{p}<\infty$, then $\Var L_{n}=o(n^{3-p})$ for $p \in [2,3)$ and $\Var L_{n}=O(\log n)$ for $p \ge 3$, as $n\to \infty$.
\end{corollary}

The asymptotics of $\E \minlength_{n}$ in  \eqref{eq:prop_first_moment} have been known before and follow, for example, from Theorem~1.8 in~\cite{McRedmond+Wade:2018} and the fact that $\minlength_{n} \od \majlength_{n}$. The main results here are about the asymptotic behavior of $\Var \minlength_{n}$ and the corollary on the behavior of $\Var L_{n}$. A recent weaker result, Theorem~6.2.6 in~\cite{McRedmond:Thesis}, asserts that $\Var L_{n}$ grows sub-polynomially if $\xi$ is centered and bounded.

\vspace{.1cm}
In view of our discussion about the joint law of $(\minlength_{n},\majlength_{n})$ in Remark~\ref{rem:non-joint}, it is not clear if, under the assumption $\E \xi^{2} < \infty$,  \eqref{eq:prop_variance_log_asymp} could lead to the stronger result $\Var L_{n} \simeq c \log n$ for some constant $c>0$ as suggested by Conjecture 1.13 in~\cite{McRedmond+Wade:2018}. However, by \eqref{eq:prop_variance_large},  $\Var \minlength_{n}$ has non-logarithmic behavior when $\E |\xi|^{3}=\infty$, and the same is very plausible for $\Var L_{n}$.

\begin{remark}\rm
In view of the functional convergence~\eqref{eq:clt_for_rw}, it is natural to ask whether the above theorems can be obtained by applying the continuous mapping theorem to the functional $F^{\smile}\!:D[0,1] \to (0,\infty)$ which assigns to each function $f$ in $D[0,1]$ the length of its convex minorant. An approach of this kind has been applied in~\cite{Lo+McRedmond+Wallace},\cite{McRedmond+Wade:2018},\cite{Wade+Xu:2015a} 
to various functionals of convex hulls of multidimensional random walks. As we are interested in the \emph{graph} of a one-dimensional random walk, this functional limit approach cannot be used here because time and space are scaled by different sequences in~\eqref{eq:clt_for_rw}. In fact, in the cases (A) and (B), the scaling in time (which is $n$) is stronger than the scaling in space (which is $a_{n}$ and thus regularly varying with index $1/\alpha$, with $\alpha=2$ in case (A)), whereas in case (C) the scaling in space is the stronger one.

\vspace{.1cm}
There is just one ``critical'' case where the two scalings coincide and the functional limit approach does work. Assume that the law of $\xi$ is such that
\begin{equation}\label{eq:clt_for_rw_alpha_{1}}
\left(\frac{S(nt)}{a_{n}}\right)_{t\in [0,1]}\ \xRightarrow{n\to\infty}\ \left(\mathcal{S}_{1}(t)\right)_{t\in [0,1]}
\end{equation}
in the Skorokhod space $D[0,1]$ with the standard $J_{1}$-topology, where $a_{n}/n \to c\in (0,\infty)$ and $(\mathcal{S}_{1}(t))_{t\in[0,1]}$ is the standard symmetric Cauchy process. Since the functional $F^{\smile}$ is continuous on a set of measure $1$ with respect to the law of the Cauchy process, the continuous mapping theorem implies that
$$
\frac {\minlength_{n}}{n}  \todistr \minlength_\infty(c),
$$
where $\minlength_\infty(c)$  is the length of the convex minorant of the Cauchy process $(c\,\mathcal{S}_{1}(t))_{t\in [0,1]}$. Note that the above argument does not completely cover the domain of attraction of the symmetric Cauchy distribution. Even if we assume that there is no centering, the sequence $a_{n}$ in~\eqref{eq:clt_for_rw_alpha_{1}} is in general of the form $a_{n} = n \ell(n)$ with some slowly varying function $\ell$.

\end{remark}


\section{Proofs explained} \label{sec:uniform_perm}

\subsection{Connection with uniform random permutations}\label{subsec:uniform_perm}
Our approach relies crucially on the following representation of the convex minorant of a random walk, observed already in 1950th by Sparre Andersen~\cite{Sparre Andersen:1954}. The version presented below is borrowed from \cite[Thms.~1 and 2]{Abramson+Pitman+Ross+Bravo:2011}, see also \cite[Thm.~1.1]{Abramson+Pitman:2011}, and valid under the assumption that the law of the increment $\xi$ is continuous.

\vspace{.1cm}
Set $[n]:=\{1,2,\ldots,n\}$ and let $\Pi_{n}$ be a permutation of $[n]$ picked uniformly at random from the symmetric group $\mathfrak{S}_{n}$, that is
$$
\P\{\Pi_{n}=\pi\}\,=\,\frac{1}{n!},\quad \pi \in\mathfrak{S}_{n}.
$$
Denote by $Z_{n,1},Z_{n,2},\ldots, Z_{n,K_{n}}$ the nonincreasingly ranked cycle lengths of $\Pi_{n}$, with $K_{n}$ being the total number of cycles. The convex minorant $t\mapsto  \minorant_{n}(t)$ being a piecewise linear function, let  $F_{n}$ denote the number of intervals where it is linear. Denote by $C_{n,1},\ldots,C_{n,F_{n}}$ the nonincreasingly ordered lengths of these intervals (on the horizontal axis).
Then the basic result we shall rely on states that
$$
(F_{n}, C_{n,1},\ldots,C_{n,F_{n}},0,0,\ldots)\ \od\ (K_{n}, Z_{n,1},Z_{n,2},\ldots, Z_{n,K_{n}},0,0,\ldots).
$$
Furthermore, given $(F_{n}, C_{n,1},\ldots,C_{n,F_{n}})$, the increments of the convex minorant over the linearity intervals are conditionally independent and the conditional law of any such increment over an interval of length $\ell$ equals the law of $S_{\ell}$. In what follows, we formally put $Z_{n,k}:=0$ for $k> K_{n}$.


\vspace{.1cm}
For $j\in [n]$, let $K_{n,j}$ be the number of cycles of length $j$ in $\Pi_{n}$, that is
$$
K_{n,j}\,:=\,\#\{k:Z_{n,k}=j\},\quad j=1,\ldots,n.
$$
Note that
$$
\quad \sum_{k=1}^{n}Z_{n,k}\,=\,n
\quad
K_{n}\,=\,\sum_{j=1}^{n}K_{n,j},
\quad\text{and}\quad
\sum_{j=1}^{n}jK_{n,j}\,=\,n.
$$
From the above observations, we immediately derive two equivalent distributional representations for the length of the convex minorant, namely
\begin{equation}\label{eq:basic_representation1}
\minlength_{n}-n\ \od\ \sum_{k=1}^{n}\left(\sqrt{Z_{n,k}^{2}+S^{2}_{k,Z_{n,k}}}-Z_{n,k}\right),
\end{equation}
and
\begin{equation}\label{eq:basic_representation2}
\minlength_{n}-n\ \od\ \sum_{j=1}^{n}\sum_{i=1}^{K_{n,j}}\left(\sqrt{j^{2}+S^{2}_{i,j}}-j\right),
\end{equation}
where the $S_{i,j}$ for $i \in \N$ and $j \in \N_{0}$ are independent random variables that are also independent of $(Z_{n,k})_{n,k \in \N}$ and satisfy $S_{i,j}\od S_{j}$ for all $i$ and $j$. Note that the summand $-n$ on the left-hand side (matched by the summands $-Z_{n,k}$ and $-j$, respectively, on the right-hand sides) corresponds to the length of the horizontal interval $[0,n]$ and should be viewed as a very rough first order approximation to the total length $\minlength_{n}$ in the cases (A) and~(B).

\vspace{.1cm}
The following smoothing argument shows that \eqref{eq:basic_representation1} and~\eqref{eq:basic_representation2} do not require that the random walk has continuous increment law. In other words:

\begin{center}
{\it The representations~\eqref{eq:basic_representation1} and~\eqref{eq:basic_representation2} remain valid without the continuity assumption.}
\end{center}

Fixing any $n\in\N$, consider the random walk $S_{k;\eps} := \xi_{\eps,1}+\cdots+\xi_{\eps,1}$ for $1\le k\le n$ and any $\eps>0$, where $\xi_{\eps,k}:= \xi_{k} +\eps N_{k}$ and $(N_{k})_{k\in\N}$ are i.i.d.\ standard normal random variables independent of $(\xi_{k})_{k\in\N}$. Let $(S_{\eps}(t))_{t\in [0,n]}$ be its linear interpolation, defined the same way as $S(t)$ above. The distribution of $\xi_{\eps,1}$ is continuous, hence the representations~\eqref{eq:basic_representation1} and~\eqref{eq:basic_representation2} hold for $\minlength_{\eps,n}$, the length of the convex minorant of $(S_{\eps}(t))_{t\in [0,n]}$. As $\eps\downarrow 0$, the process $(S_\eps(t))_{t\in [0,n]}$ converges to $(S(t))_{t\in [0,n]}$ weakly in the space $C[0,n]$, and since the functional assigning to each continuous function the length of its convex minorant is continuous on $C[0,n]$ (by the Cauchy--Crofton formula), the continuous mapping theorem implies that $\minlength_{n;\eps}$ converges in distribution to $\minlength_{n}$, as $\eps\downarrow 0$. Finally, the claim follows because the right-hand sides of~\eqref{eq:basic_representation1} and~\eqref{eq:basic_representation2}  for $(S_\eps(t))_{t\in [0,n]}$ converge, as $\eps\downarrow 0$, to the corresponding expressions for $(S(t))_{t\in [0,n]}$.

\subsection{Explanation of the proofs in the cases (A) and (B)} \label{subsec:explanation}
The structure of uniform random permutations is well understood. In particular, see Theorem 1.3. in \cite{Arratia+Barbour+Tavare:2003}, it is known that
\begin{equation}\label{eq:cycles_convergence1}
\left(K_{n,1},K_{n,2},\ldots,K_{n,n},0,0,\ldots\right)\ \todistr\ (P_{1},P_{2},P_3,\ldots),
\end{equation}
where the $P_{j}$ are mutually independent and the law of $P_{j}$ is Poisson with mean $1/j$. Moreover, the convergence is fast: there exists a coupling (called the {\it Feller coupling}) such that
\begin{equation}\label{eq:L_{1}_conv_marginal}
\E |K_{n,j}-P_{j}|\ \le\ \frac{2}{n+1},\quad j=1,\ldots,n,
\end{equation}
and thus
\begin{equation}\label{eq:L_{1}_conv_{j}oint}
\E \sum_{j=1}^{n} |K_{n,j}-P_{j}|\ <\ 2,
\end{equation}
see \cite[Remark on p.~18 and Eq.~(1.26)]{Arratia+Barbour+Tavare:2003}. This coupling will be crucial for the proof in case (A).  Put
$$
Y_{j}\,:=\,\sum_{i=1}^{P_{j}}\left(\sqrt{j^{2}+S^{2}_{i,j}}-j\right)
$$
for $j\in\N$. Formula \eqref{eq:L_{1}_conv_{j}oint} in conjunction with representation \eqref{eq:basic_representation2} strongly suggests that the asymptotic behavior of $\minlength_{n}-n$ should be well approximated by that of the sum
$$
V_{n}\,:=\,\sum_{j=1}^{n}Y_{j},
$$
which is simply a partial sum of independent (but not identically distributed) random variables. The corresponding limit laws are usually called distributions of class $L$, but in our case (A) the limit turns out to be normal.

\vspace{.1cm}
Below we will prove Theorem \ref{thm:finite_variance} by showing that the distributions of normalized random variables $\minlength_{n}-n$ and $V_{n}$ are asymptotically close, and then checking the classical conditions for convergence in distribution of the $V_{n}$ after suitable normalization, which are the row sums of triangular arrays whose rows consist of independent random variables.

An interesting observation is that the above argument, based on replacing $K_{n,j}$ by $P_{j}$ in representation \eqref{eq:basic_representation2}, fails to work in the cases (B) and (C). In order to heuristically explain our approach in case (B), we recall another classical fact from the theory of random permutations, namely (see Vershik and Schmidt \cite{Vershik+Schmidt:1977} or Kingman \cite{Kingman:1977})
\begin{equation}\label{eq:cycles_convergence2}
\left(\frac{Z_{n,1}}{n},\frac{Z_{n,2}}{n},\ldots\right)\ \todistr\ \left(Z_{1},Z_{2},\ldots\right),
\end{equation}
where the random sequence $(Z_{1},Z_{2},\ldots)$ has a Poisson--Dirichlet distribution with parameter $\theta=1$. Furthermore, there exists a coupling such that
\begin{equation}\label{eq:L_{1}_conv_{j}oint_large_cycles}
\E\left(\sum_{k=1}^{\infty}\left|\frac{Z_{n,k}}{n}-Z_{k}\right|\right)\ \simeq\ \frac{\log n}{4n}\quad\text{as }n\to\infty,
\end{equation}
see \cite[Theorem 8.10]{Arratia+Barbour+Tavare:2006}. Put $b_{n}:=a_{n}^{2}/n$ and note that regular variation of $(a_{n})_{n\in\N}$ with index $\frac{1}{\alpha}$ implies regular variation of the sequence $(b_{n})_{n\in\N}$ with index $\frac{2}{\alpha}-1>0$. In particular, $b_{n}\to\infty$ if $\alpha<2$. Recalling relation \eqref{eq:basic_representation1}, we can argue heuristically as follows:
\begin{align*}
\frac{\minlength_{n}-n}{b_{n}}\ &\od\ \frac{n}{a_{n}^{2}}\sum_{k=1}^{n}\left(\sqrt{Z_{n,k}^{2}+S_{k,Z_{n,k}}^{2}}-Z_{n,k}\right)\\
&=\ \frac{n}{a_{n}^{2}}\sum_{k=1}^{n}Z_{n,k}\left(\sqrt{1+\frac{S_{k,Z_{n,k}}^{2}}{Z_{n,k}^{2}}}-1\right)\\
&\approx\ \frac{n}{a_{n}^{2}}\sum_{k=1}^{n}\frac{S_{k,Z_{n,k}}^{2}}{2Z_{n,k}}\\
&\approx\ \frac{1}{a_{n}^{2}}\sum_{k=1}^{n}\frac{S_{k,[nZ_{k}]}^{2}}{2Z_{k}}\ \todistr\ \frac{1}{2}\sum_{k=1}^{\infty}\frac{(\mathcal{S}_{\alpha}^{(k)}(Z_{k}))^{2}}{Z_{k}},
\end{align*}
where the first approximation stems from the one-term Taylor expansion, the second is a consequence of \eqref{eq:L_{1}_conv_{j}oint_large_cycles} and the asserted convergence follows from \eqref{eq:clt_for_rw}, the $\mathcal{S}_{\alpha}^{(k)}$ being independent copies of $\mathcal{S}_{\alpha}$ which are also independent of $(Z_{1},Z_{2},\ldots)$. Proposition \ref{prop:limit_is_finite} below will show that the series in the last line is a.s.~finite. Moreover, the above heuristic turns out to be correct and this will provide the proof of Theorem~\ref{thm:infinite_variance}.

\section{Proofs}

\subsection{Proof of Theorem \ref{thm:finite_variance}}

Let $\tminlength_{n}$ denote the random variable on the right-hand side of \eqref{eq:basic_representation2} increased by~$n$ so that $\tminlength_{n} \od \minlength_{n} $. 

As discussed in Section~\ref{subsec:explanation}, the first step of the proof is to show that
\begin{equation}\label{eq:thm1_proof1}
(\tminlength_{n}-n-V_{n})/\sqrt{\log n}\ \toprobab\ 0.
\end{equation}
We have
\begin{equation}\label{eq:lem_ui_proof00}
\sqrt{j^{2}+S_{1,j}^{2}}-j\ =\ \frac{S_{1,j}^{2}}{j+\sqrt{j^{2}+S_{1,j}^{2}}}\ \le\ \frac{S_{1,j}^{2}}{2j},
\end{equation}
in particular
\begin{equation}\label{eq:first_moment_estimate1}
\E \left(\sqrt{j^{2}+S_{1,j}^{2}}-j\right)\ \le\ \frac{\E S_{1,j}^{2}}{2j}\ =\ \frac{\sigma^{2}}{2}.
\end{equation}
Using the definition of $V_{n}$, this entails
\begin{align*}
\E \left|\tminlength_{n}-n-V_{n}\right|\ &\le\ \sum_{j=1}^{n}\E\left|\sum_{i=1}^{K_{n,j}}\left(\sqrt{j^{2}+S^{2}_{i,j}}-j\right)-\sum_{i=1}^{P_{j}}\left(\sqrt{j^{2}+S^{2}_{i,j}}-j\right) \right|\\
&=\ \sum_{j=1}^{n}\E |K_{n,j}-P_{j}|\,\E \left(\sqrt{j^{2}+S_{j}^{2}}-j\right)\ \le\ \sigma^{2}\ <\ \infty,
\end{align*}
where the penultimate inequality follows from  \eqref{eq:L_{1}_conv_{j}oint} and \eqref{eq:first_moment_estimate1}. Now \eqref{eq:thm1_proof1} follows from the last line and the Markov inequality.

\vspace{.1cm}
The second step is to simplify $V_{n}$. Put
\begin{equation} \label{eq:W_{n} definition}
W_{n}\ :=\ \sum_{j=1}^{n}\left(\sqrt{j^{2}+S_{1,j}^{2}}-j\right)\ind_{\{P_{j}=1\}},\quad n\in\N.
\end{equation}
Observe that
$$
\sum_{j=1}^{\infty}\ind_{\{P_{j}\ge 2\}}<\infty\quad\text{a.s.}
$$
by the Borel--Cantelli lemma. Then
\begin{align*} 
0\ \le\ V_{n}-W_{n}\ &=\ \sum_{j=1}^{n}\sum_{i=1}^{P_{j}}\left(\sqrt{j^{2}+S^{2}_{i,j}}-j\right)\ind_{\{P_{j}\ge 2\}} \notag\\
&\le\ \sum_{j=1}^{\infty}\sum_{i=1}^{P_{j}}\left(\sqrt{j^{2}+S^{2}_{i,j}}-j\right)\ind_{\{P_{j}\ge 2\}}\ <\ \infty\quad\text{a.s.},
\end{align*}
which implies that $(V_{n}-W_{n})_{n\in\N}$ is bounded in probability. Hence, by \eqref{eq:thm1_proof1}, it suffices to prove the theorem for $W_{n}$ instead of $\minlength_{n} - n$.

\vspace{.1cm}
The third step is to simplify $W_{n}$. Let us rewrite it as
\begin{align}  \label{eq:W_{n}_Taylor}
W_{n}\ &=\ \sum_{j=1}^{n}j\left(\sqrt{1+\frac{S_{1,j}^{2}}{j^{2}}}-1\right)\ind_{\{P_{j}=1\}} \notag\\
&=\ \sum_{j=1}^{n}j\left(\frac{S_{1,j}^{2}}{2j^{2}}- \frac{S_{1,j}^{4}}{8j^{4}} \cdot \theta \left(\frac{S_{1,j}^{2}}{j^{2}}\right)\hspace{-1.2mm}\right)\ind_{\{P_{j}=1\}},
\end{align}
where $\theta:[0,\infty) \to [0,1]$ is a continuous bounded function resulting from Lagrange's form of the remainder in the Taylor expansion of $x\mapsto \sqrt{1+x}$ at $x=0$.
We claim that
\begin{equation}\label{eq:finite_variance_remainder_{n}egligible}
\frac{1}{\sqrt{\log n}}\sum_{j=1}^{n}\frac{S_{1,j}^{4}}{j^3}\ind_{\{P_{j}=1\}}\ \toprobab\ 0,
\end{equation}
for which it obviously suffices to verify that
\begin{equation}\label{eq:finite_variance_remainder_{n}egligible_aux}
\sum_{j=1}^{\infty}\frac{S_{1,j}^{4}}{j^3}\ind_{\{P_{j}=1\}}\ <\ \infty \quad \text{a.s}.
\end{equation}
Put $p_{j}:=\P\{P_{j}=1\}$ and note that
\begin{equation}\label{eq:P_{j}_asymp}
\frac{1}{j}\ \ge\ p_{j}\ =\ e^{-1/j}\frac{1}{j}\ =\ \frac{1}{j}+O\left(\frac{1}{j^{2}}\right)\quad\text{as }j\to\infty.
\end{equation}

As $\E\xi=0$ and $\sigma^{2}=\E\xi^{2}<\infty$, it follows by \cite[Thm.~10.2 on p.~46]{Gut:2009} (or \cite[Thm.~4]{Davis:1968}) that
\begin{equation}\label{eq:LIL_convergence_rate}
\sum_{j=1}^\infty \frac 1j \P\left\{|S_{j}| > a\sigma\sqrt {j\log \log j}\right\}\ <\ \infty
\end{equation}
for any $a> \sqrt 2$. By combining this and the Borel--Cantelli lemma, only finitely many events $\{|S_{1,j}| > 2\sigma\sqrt {j\log \log j}, P_{j}=1\}$, $j\in\N$, occur with probability~$1$. 
Now the proof of~\eqref{eq:finite_variance_remainder_{n}egligible_aux}, which in turn implies~\eqref{eq:finite_variance_remainder_{n}egligible}, can be completed by using 
the inequality
\[
\sum_{j=1}^{\infty}\frac{S_{1,j}^{4}}{j^3}\ind_{\{P_{j}=1\}}\ \le\
\sum_{j=1}^{\infty}\frac{S_{1,j}^{4}}{j^3}\ind_{\{|S_{1,j}| > 2\sigma\sqrt {j\log \log j}, P_{j}=1\}}\ +\ \sum_{j=1}^{\infty}\frac{16\sigma^{4}}{j} (\log \log j)^{2}\ind_{\{P_{j}=1\}}.
\]
The first sum on the right-hand side is finite a.s.\ since it contains a.s.\ only finitely many non-zero terms, and the second sum is finite a.s.\ because it has finite expectation.


In view of \eqref{eq:W_{n}_Taylor} and \eqref{eq:finite_variance_remainder_{n}egligible}, it suffices to prove Theorem \ref{thm:finite_variance} with $\minlength_{n}-n$ replaced by
$$
W_{n}'\ :=\ \sum_{j=1}^{n}\frac{S_{1,j}^{2}}{2j}\ind_{\{P_{j}=1\}}.
$$
Therefore, as the fourth step of the proof, we will show that
\begin{equation}\label{eq:W_{n}_prime_CLT}
\frac{1}{\sqrt{\log n}}\left(W_{n}' - \sum_{j=1}^{n}\frac{\sigma_{j}^{2}}{2j}\right)\ \todistr\ \mathcal{N}\left(0,\frac{3\sigma^{4}}{4}\right).
\end{equation}
Even though the variables $W_{n}'$ are sums of independent random variables of a rather simple structure, it is not easy to obtain a central limit theorem for $W_{n}'$ without additional moment assumptions. For example, when trying to verify the Lindeberg condition, the fourth moment of $S_{1,j}$ appears, which is not assumed to be finite here. 

We use a general result \cite[Thm.~18 in Chap.~IV, $\S$4]{Petrov:1975}, which ensures convergence of row sums in triangular arrays of random variables to a normal distribution. According to the three conditions of the theorem, we need to check that
\begin{equation}\label{eq:infinitely_small}
\lim_{n\to\infty}\sum_{j=1}^{n}\P\left\{\frac{S_{1,j}^{2}}{2j}\ind_{\{P_{j}=1\}}\ge \varepsilon\sqrt{\log n}\right\}\ =\ 0 \quad \text{for every } \varepsilon >0,
\end{equation}
\begin{equation} \label{eq:variance_converges}
\lim_{n\to\infty} \frac{1}{\log n} \sum_{j=1}^{n} \Var \left(\frac{S_{1,j}^{2}}{2j}\ind_{\{P_{j}=1,S_{1,j}^{2}/(2j) < \sqrt{\log n}\}}\right) =\frac{3\sigma^{4}}{4},
\end{equation}
and
\begin{equation}\label{eq:finite_variance_centering}
\lim_{n\to\infty}\frac{1}{\sqrt{\log n}}\left[\sum_{j=1}^{n} \E\left(\frac{S_{1,j}^{2}}{2j}\ind_{\left\{P_j=1, S_{1,j}^{2}/(2j)<\sqrt{\log n}\right\}}\right)-\sum_{j=1}^{n}\frac{\sigma_{j}^{2}}{2j} \right ] = 0.
\end{equation}

Our idea is to approximate $S_{1,j}/\sqrt{j}$ by $\sigma_j \mathcal{N}$, where $\mathcal N$ is a standard normal variable and $\sigma_{j}^{2}$ is given by~\eqref{eq:sigma_{n}_definition}. We claim that, for any $\gamma \ge 0$ and $\varepsilon >0$,
\begin{multline}\label{eq:main_approximation}
\lim_{n\to\infty} \frac{1}{(\log n)^{\gamma/2}} \sum_{j=1}^{n}\frac{1}{j} \left[ \E\left(\frac{|S_{1,j}|^{2\gamma}}{j^\gamma}\ind_{\left\{S_{1,j}^{2}/j<\varepsilon \sqrt{\log n}\right\}}\right) - \E \left(\sigma_{j}^{2 \gamma} |\mathcal{N}|^{2\gamma}\ind_{\left\{\sigma_j^2 \mathcal{N}^{2} < \varepsilon \sqrt{\log n}\right\}}\right)\right]\\
=\ 0.
\end{multline}

To prove this, we first note that 
\begin{equation} \label{eq:fourth_moment}
\E(X^\gamma \1_{\{X < a\}})\ =\ -\int_{[0,a)} x^\gamma  {\rm d} \P\{ X \ge x \} = -a^\gamma \P\{X \ge a \}+\gamma a^{\gamma}\int_0^1 t^{\gamma - 1} \P\{X \ge a t \}\ {\rm d} t
\end{equation}
for any non-negative random variable $X$ and $a, \gamma >0$. Upon applying this formula twice with $X= S_{1,j}^2/j$ and $X= \sigma_j^2 \mathcal{N}^2$, we obtain
\begin{multline*} 
 \frac{1}{a^{\gamma}} \bigg |\E  \bigg ( \frac{|S_{1,j}|^{2 \gamma}}{j^\gamma} \1_{\{ S_{1,j}^2/j < a\}}\bigg)\,-\,\E \big( \sigma_j^{2 \gamma } |\mathcal{N}|^{2 \gamma} \1_{\{ \sigma_j^2 \mathcal{N} <  a  \}}\big)\bigg| \\
\le\ \bigg | \P \bigg \{ \frac{S_{1,j}^2}{j} \ge a \bigg\}  - \P \big \{ \sigma_j^2\mathcal{N}^2 \ge a \big \} \bigg |\,+\, \gamma  \int_0^1 t^{ \gamma - 1} \bigg | \P \bigg \{ \frac{S_{1,j}^2}{j} \ge at  \bigg\}  - \P \big \{ \sigma_j^2\mathcal{N}^2 \ge at \big \} \bigg |\  {\rm d} t.
\end{multline*}
Then, taking $a  = \varepsilon \sqrt{\log n}$, we see that for every $j\leq n$,
\begin{multline} \label{eq:main_bound}
 \frac{1}{j (\log n)^{\gamma/2}}\left[ \E\left(\frac{|S_{1,j}|^{2\gamma}}{j^\gamma}\ind_{\left\{S_{1,j}^{2}/j<\varepsilon \sqrt{\log n}\right\}}\right) - \E \left(\sigma_{j}^{2 \gamma} |\mathcal{N}|^{2\gamma}\ind_{\left\{\sigma_j^2 \mathcal{N}^{2} < \varepsilon \sqrt{\log n}\right\}}\right)\right] \\
\le\ 4 \varepsilon^\gamma \cdot \frac{1 }{j} \sup_{x \in \R } \bigg | \P \bigg \{ \frac{S_{1,j}}{\sqrt j} < x \bigg\}  - \P \big \{ \sigma_j \mathcal{N} < x\big \}\bigg |. 
\end{multline}
Clearly, this bound is also valid for $\gamma =0$. On the other hand, the dominated convergence theorem ensures
\[
\lim_{n\to\infty}\E \left[\frac{1}{j (\log n)^{\gamma/2}}\left(\frac{|S_{1,j}|^{2\gamma}}{j^\gamma}\ind_{\left\{S_{1,j}^{2}/j<\varepsilon \sqrt{\log n}\right\}} - \sigma_{j}^{2 \gamma} |\mathcal{N}|^{2\gamma}\ind_{\left\{\sigma_j^2 \mathcal{N}^{2} < \varepsilon \sqrt{\log n}\right\}}\right)\right]\ =\ 0,
\]
for every $j\in\N$ because the random variable in square brackets is bounded by $2 \varepsilon^\gamma/j$. Since, furthermore, the sequence on the right-hand side of \eqref{eq:main_bound} is summable over $j\in\N$, see \cite[p.~1480]{Davis:1968} and~\cite[p.~130]{Petrov:1975}\footnote{This result was actually proved in~\cite{Friedman+Katz+Koopmans:1966} but stated there without uniformity in $x$.},
\eqref{eq:main_approximation} follows by another appeal to the dominated convergence theorem.

We can now finish the proof of \eqref{eq:W_{n}_prime_CLT}. Recall that $P_{j}$ is independent of $S_{1,j}$. Using \eqref{eq:P_{j}_asymp} and then \eqref{eq:main_approximation} with $\gamma=0$, we have
\begin{align*}
\lim_{n\to\infty}\sum_{j=1}^{n}\P\left\{\frac{S_{1,j}^{2}}{2j}\ind_{\{P_{j}=1\}}\ge \varepsilon\sqrt{\log n}\right\}\  
&=\ \lim_{n\to\infty}\sum_{j=1}^{n}\frac{1}{j}\P\left\{\frac{S_{1,j}^{2}}{2j}\ge \varepsilon\sqrt{\log n}\right\}\\
&=\ \lim_{n\to\infty}\sum_{j=1}^{n} \frac{1}{j} \P\left\{ \sigma_j^2 \mathcal{N}^2 \ge 2 \varepsilon\sqrt{\log n}\right\}.
\end{align*}
The last limit is equal to zero since $\lim_{j\to\infty}\sigma_j^2 \ = \ \sigma^2<\infty$ and in view of the following simple Markov-type inequality
\begin{equation} \label{eq:normal_truncated_bound}
\E\left( \sigma_j^{2\gamma} \mathcal{N}^{2 \gamma} \ind_{ \left \{ \sigma_j^2 \mathcal{N}^2 \ge a\right\}}\right)\ \le\ a^{-3} \sigma_j^{2\gamma + 6}\E \mathcal{N}^{2 \gamma +6} < \infty,
\end{equation}
which holds for all $j \in \N$, $\gamma \ge 0$, $a >0$, and which we applied with $\gamma=0$. This finishes the proof of \eqref{eq:infinitely_small}.

Furthermore, we have
\begin{align*}
&\sum_{j=1}^{n} \Var \left(\frac{S_{1,j}^{2}}{2j}\ind_{\{P_{j}=1,S_{1,j}^{2}/(2j) < \sqrt{\log n}\}}\right)\\
&\hspace{.2cm}=\ \sum_{j=1}^{n}p_{j}\,\E\left(\frac{S_{1,j}^{4}}{4j^{2}}\ind_{\left\{S_{1,j}^{2}/j< 2\sqrt{\log n}\right\}}\right)-\sum_{j=1}^{n} \left(p_{j}\,\E\left(\frac{S_{1,j}^{2}}{2j}\ind_{\left\{S_{1,j}^{2}/j< 2\sqrt{\log n}\right\}}\right)\hspace{-1.4mm}\right)^{2}.
\end{align*}
The second sum on the right-hand side is increasing in $n$ and bounded by $\sum_{j=1}^{\infty}\frac{\sigma^{4}}{4j^{2}}<
\infty$ in view of \eqref{eq:P_{j}_asymp}. Appealing once again to \eqref{eq:P_{j}_asymp} and using  \eqref{eq:main_approximation} with $\gamma = \varepsilon =2$, we obtain
\begin{align*}
&\hspace{-2cm}\lim_{n \to \infty} \frac{1}{\log n} \sum_{j=1}^{n} \Var \left(\frac{S_{1,j}^{2}}{2j}\ind_{\{P_{j}=1,S_{1,j}^{2}/(2j) < \sqrt{\log n}\}}\right)\\
&=\ \lim_{n\to\infty}\frac{1}{\log n}\sum_{j=1}^{n}p_{j}\,\E\left(\frac{S_{1,j}^{4}}{4j^{2}}\ind_{\left\{S_{1,j}^{2}/j< 2\sqrt{\log n}\right\}}\right)\\
&=\ \lim_{n\to\infty}\frac{1}{\log n}\sum_{j=1}^{n}\frac{1}{j}\,\E\left(\frac{S_{1,j}^{4}}{4j^{2}}\ind_{\left\{S_{1,j}^{2}/j< 2\sqrt{\log n}\right\}}\right)\\
&=\ \lim_{n \to \infty} \frac{1}{4 \log n}  \sum_{j=1}^{n} \frac{1}{j} \,\E\left( \sigma_j^4 \mathcal{N}^4 \ind_{ \left \{ \sigma_j^2 \mathcal{N}^2 < 2\sqrt{\log n}\right\}}\right)\\  
&=\ \lim_{n \to \infty} \frac{1}{4 \log n}  \sum_{j=1}^{n} \frac{3 \sigma_j^4 }{j}  \ = \ \frac{3 \sigma^4}{4}
\end{align*}
using that the family $(\sigma_j^4 \mathcal{N}^{4})_{j\in\N}$ is uniformly integrable by $\sigma_{j}^{2}\to \sigma^{2}$. This proves \eqref{eq:variance_converges}.

Finally, from  \eqref{eq:main_approximation} and \eqref{eq:normal_truncated_bound}  with $\gamma =1$, $ \varepsilon =2$, and $a=2 \sqrt{\log n}$, 
\begin{align*}
& \lim_{n\to\infty}\frac{1}{\sqrt{\log n}}\left[\sum_{j=1}^{n} \E\left(\frac{S_{1,j}^{2}}{2j}\ind_{\left\{P_j=1, S_{1,j}^{2}/(2j)<\sqrt{\log n}\right\}}\right)-\sum_{j=1}^{n}\frac{\sigma_{j}^{2}}{2j} \right ] \\
& = \ \lim_{n\to\infty}\frac{1}{2 \sqrt{\log n}} \sum_{j=1}^{n} \frac{1}{j} \left[ \E\left(\sigma_j^2 \mathcal{N}^2 \ind_{\left\{\sigma_j^2 \mathcal{N}^2<2\sqrt{\log n}\right\}}\right)-\sigma_{j}^{2}\right ] \ =\ 0.
\end{align*}
This proves \eqref{eq:finite_variance_centering}. The proof of~\eqref{eq:W_{n}_prime_CLT} is herewith complete. Thus, we established \eqref{eq:clt_finite_variance}.

\vspace{.1cm}
If $\E\xi^{2}\log^{+} |\xi|<\infty$, then $\sigma_{j}^{2}$  can be replaced by $\sigma^{2}$ throughout the proof, implying~\eqref{eq:clt_finite_variance2}. The key observation is that the modified sequence on the right-hand side of estimate \eqref{eq:main_bound}  remains summable over $j\in\N$ by
Lemma 1 and the Theorems on p.~1480 in \cite{Davis:1968}.  This finishes the proof of Theorem \ref{thm:finite_variance}.

\subsection{Proof of Theorem~\ref{thm:moments}}

We first give explicit formulae for $\E \minlength_{n}$ and $\Var \minlength_{n}$ and put for  simplicity
$$
\eta_{j}\,:=\,\sqrt{j^{2}+S_{j}^{2}}-j, \quad j\in\N.
$$
\begin{lemma}\label{lem:moments}
If $\E\xi=0$ and $\sigma^{2}=\E\xi^{2}<\infty$, then
$$
\E \minlength_{n} - n\ =\ \sum_{j=1}^{n}\frac{\E\eta_{j}}{j}\quad\text{and}\quad\Var \minlength_{n}\ =\ \sum_{j=1}^{n}\frac{\E \eta_{j}^{2}}{j}\,+\,O(1)\quad\text{as }n\to\infty.
$$
\end{lemma}

\begin{proof}
It will be used that for any integers $1\le j \neq k \le n$,
\begin{equation} \label{eq:K_{n}j_moments}
\E K_{n,j}\ =\ \frac1j, \quad \E (K_{n,j}(K_{n,j}-1))\ =\ \frac{\1_{\{2j \le n\}}}{j^{2}}\quad
\text{and}\quad\E (K_{n,j}K_{n,k})\ =\ \frac{\1_{\{j+k \le n\}}}{jk};
\end{equation}

see \cite[Lemma~1.1]{Arratia+Barbour+Tavare:2003}. Thus, $K_{n,j}$ and $K_{n,k}$ are uncorrelated whenever $j+k \le n$.

\vspace{.1cm}
The formula for $\E \minlength_{n}$ follows immediately from representation~\eqref{eq:basic_representation2} and~\eqref{eq:K_{n}j_moments}. For the variance of $\minlength_{n}$, we obtain with the help of the formula for the variance of a random sum of i.i.d. random variables:
\begin{align*}
\Var&\Bigg(\sum_{i=1}^{K_{n,j}}\Big(\sqrt{j^{2}+S^{2}_{i,j}}-j\Big) \Bigg)\ =\ \E K_{n,j}\Var\eta_{j}+(\E\eta_{j})^{2} \Var K_{n,j}\\
&=\ \frac{\E \eta_{j}^{2}}{j}-\frac{(\E\eta_{j})^{2}}{j}+(\E\eta_{j})^{2}\left(\frac{j-\1_{\{2j>n\}}}{j^{2}}\right)\ =\ \frac{\E \eta_{j}^{2}}{j}-\frac{(\E\eta_{j})^{2} \1_{\{2j> n\}}}{j^{2}}.
\end{align*}
Similarly, by conditioning on $K_{n,j}$ and $K_{n,k}$ and setting $\gamma_{jk}^{(n)}:=\Cov( K_{n,j}, K_{n,k})$,
$$
\Cov\Bigg(\sum_{i=1}^{K_{n,j}}\Big(\sqrt{j^{2}+S^{2}_{i,j}}-j\Big),\sum_{i=1}^{K_{n,k}}\Big(\sqrt{k^{2}+S^{2}_{i,k}}-k\Big)\hspace{-3pt}\Bigg)\ =\ \gamma_{jk}^{(n)}\,\E \eta_{j}\,\E\eta_{k}.
$$
Combining these formulas with~\eqref{eq:basic_representation2} and \eqref{eq:K_{n}j_moments}, we obtain
\begin{align*}
\Var(\minlength_{n})\ &=\ \sum_{j=1}^{n} \Var \bigg( \sum_{i=1}^{K_{n,j}}\Big(\sqrt{j^{2}+S^{2}_{i,j}}-j\Big) \bigg) + \sum_{1\le j\neq k\le n}\gamma_{jk}^{(n)}\,\E \eta_{j}\,\E\eta_{k}\notag\\
&=\ \sum_{j=1}^{n}\frac{\E \eta_{j}^{2}}{j} -  \sum_{\substack{1\le j\le n: \\2j >n}} \frac{(\E \eta_{j})^{2}}{j^{2}} - \sum_{\substack{1 \le j\neq k \le n:\\ j+k >n}} \frac{\E\eta_{j}\,\E\eta_{k}}{jk}\\
&=\ \sum_{j=1}^{n}\frac{\E \eta_{j}^{2}}{j} - \sum_{\substack{1 \le j, k \le n:\\ j+k >n}} \frac{\E \eta_{j}\,\E \eta_{k}}{jk}.
\end{align*}
The last term on the right-hand side can be estimated by using inequality \eqref{eq:first_moment_estimate1}, viz.
\begin{align*}
0\ &\le\ \sum_{\substack{1 \le j, k \le n:\\j+k >n}}\frac{\E \eta_{j}\,\E \eta_{k}}{jk}\ \le\   \frac{\sigma^{4}}{4}\sum_{\substack{1\le j, k \le n:\\ j+k >n}}\frac{1}{jk}\ =\ \frac{\sigma^{4}}{4}\sum_{j=1}^{n}\frac{1}{j}\sum_{k=n-j+1}^{n}\frac{1}{k}\\
&=\ \frac{\sigma^{4}}{4}\sum_{j=1}^{n-1}\frac{1}{j}\sum_{k=n-j+1}^{n}\frac{1}{k}\ +\ O\left(\frac{\log n}{n}\right)\\
&\le\ \frac{\sigma^{4}}{4}\sum_{j=1}^{n-1}\frac{1}{j}\sum_{k=n-j+1}^{n}\int_{k-1}^{k}\frac{{\rm d}x}{x}\ +\ O\left(\frac{\log n}{n}\right)\\
&=\ \frac{\sigma^{4}}{4}\sum_{j=1}^{n-1}\frac{\log n-\log (n-j)}{j}\ +\ O\left(\frac{\log n}{n}\right)\\
&=\ -\frac{\sigma^{4}}{4n}\sum_{j=1}^{n-1}\left(\frac{j}{n}\right)^{-1}\log\left(1-\frac{j}{n}\right)\ +\ O\left(\frac{\log n}{n}\right)\\
&=\ -\frac{\sigma^{4}}{4}\int_{0}^{1} x^{-1}\log (1-x){\rm d}x\ +\ O\left(\frac{\log n}{n}\right),
\end{align*}
as $n \to \infty$, where the last passage follows from formula (21) in \cite{IksMarMoe:2009} with $b=k=1$. This proves the formula for $\Var \minlength_{n}$ because the last integral is finite.
\end{proof}

We are ready to prove our claims on the asymptotics of the moments of $\minlength_{n}$.

\begin{proof}[Proof of Theorem~\ref{thm:moments}]
By the equality in \eqref{eq:lem_ui_proof00}, the law of large numbers and the central limit theorem,
\begin{equation}\label{eq:lem_ui_proof1}
\eta_{n}\ \todistr\ \frac{1}{2}\,\mathcal{N}^{2}(0,\sigma^{2}).
\end{equation}
Under the assumptions $\E \xi^{2}<\infty$ and $\E\xi=0$, the family $(S_{n}^{2}/n)_{n\in\N}$ is uniformly integrable \cite[Thm.~1.6.3]{Gut:2009} whence, by the inequality in \eqref{eq:lem_ui_proof00}, the same holds for $(\eta_{n})_{n\in\N}$. In conjunction with \eqref{eq:lem_ui_proof1}, this yields
\begin{equation}\label{eq:lem_ui_proof0}
\lim_{n\to\infty}\E \eta_{n}\ =\ \frac{1}{2} \E \mathcal{N}^{2}(0,\sigma^{2})\ =\ \frac{\sigma^{2}}{2}.
\end{equation}
Combined with Lemma~\ref{lem:moments}, this gives the asympotics of $\E \minlength_{n} -n$ stated in~\eqref{eq:prop_first_moment}. 

For the proof of the remaining claims, let us introduce the function
\[
g(x)\,:=\,x^2/2 - \sqrt{1+x^2}+1, \quad x \ge 0,
\]
which satisfies $g(x) \simeq x^4/8$ as $x \to 0$, and is non-negative and strictly increasing for $x>0$ since $g'(x)>0$. We have
\begin{equation} \label{eq:E_eta}
\frac{\sigma^{2}}{2} -\,\E \eta_{j}\ =\ \E\left(\frac{S_{j}^{2}}{2j} - \Big( \sqrt{j^{2}+S_{j}^{2}}-j \Big) \right ) \ =\ j\E g\left(\frac{|S_{j}|}{j} \right ), 
\end{equation}
hence by Lemma~\ref{lem:moments}, 
$$
\E \minlength_{n} - n = \sum_{j=1}^{n} \frac{\E \eta_{j}}{j} = \sum_{j=1}^{n} \frac{\sigma^{2}}{2j}\,-\, \sum_{j=1}^{n} \E g\left(\frac{|S_{j}|}{j} \right ).
$$

To study convergence of the last sum, we employ  \cite[Thm.~1]{Pruss:1997} which reads as follows: there exist positive constants $C_1, C_2$ such that for every $x >0$, 
\begin{equation} \label{eq:Pruss}
C_1 x^{-2} \E (\xi^2 \1_{\{|\xi| \ge x\}})\ \le \ F(x):=\sum_{j=1}^\infty \P\{|S_j| \ge x j\} \ \le \ C_2 x^{-2} \E (\xi^2 \1_{\{|\xi| \ge x\}}).
\end{equation}
By Fubini's theorem, we have
\[
\sum_{j=1}^\infty \E g(|S_j|/j)\ =\ \int_0^\infty g'(x) F(x)\ {\rm d} x.
\]
Therefore, by \eqref{eq:Pruss} the right-hand side is finite if and only if the integrals
\begin{equation} \label{eq:Pruss_Tonelli}
\int_0^\infty \frac{g'(x)}{x^2} \E (\xi^2 \1_{\{|\xi| \ge x\}})\ {\rm d} x\ =\ \int_0^\infty t^2 \left( \int_0^t \frac{g'(x)}{x^2} {\rm d} x \right)\ \P\{ |\xi| \in {\rm d } t \}
\end{equation}
are finite. The function $g'(x) /x^2$ is integrable at $x=0$ since $g'(x) \simeq x^3/2$ as $x \to 0$. On the other hand, since $g'(x) = x - 1 + o(1)$ as $x \to \infty$, we have $\int_0^t ( g'(x)/x^2) {\rm d} x = \log t + O(1)$ as $t \to \infty$. Hence, by $\E \xi^2 < \infty$, the right-hand side of \eqref{eq:Pruss_Tonelli} is finite if and only if $\E\xi^{2}\log^{+} |\xi|<\infty$, which is thus equivalent to  the asymptotic relation $\E \minlength_{n} = n + \frac12 \sigma^2 \log n + O(1)$ as $n \to \infty$. 

\vspace{.1cm}
Let us now prove the claims about $\Var \minlength_{n}$. Note that, by \eqref{eq:E_eta}, 
\[
\E\eta_{j}^{2}\ =\ \E\left(\sqrt{j^{2}+S_{j}^{2}}-j\right)^{2}
\ =\ \E \left(S_{j}^{2}-2j\left(\sqrt{j^{2}+S_{j}^{2}}-j\right)\right)\ =\ 2j^2 \E g\left(\frac{|S_{j}|}{j} \right ).
\]
Thus, by Lemma~\ref{lem:moments},
\begin{equation} \label{eq:Var=}
\Var(\minlength_{n})\ =\ \sum_{j=1}^{n}\frac{\E\eta_{j}^{2}}{j}\,+\,O(1)\ =\
2 \sum_{j=1}^{n} j \E g\left(\frac{|S_{j}|}{j} \right ) + O(1) \quad\text{as }n\to\infty.
\end{equation}
In particular, by \eqref{eq:lem_ui_proof0} and \eqref{eq:E_eta}, this implies $\Var(\minlength_{n}) = o(n)$, that is, \eqref{eq:prop_variance_poly_asymp} holds for $p=2$.

From now on we assume $\E |\xi|^p< \infty$ for some $p \in (2,3]$. To prove the remaining claims we first need to estimate truncated moments of $S_{j}$. Let us show that  for every $\varepsilon >0$, 
\begin{equation} \label{eq:second_moment_truncated}
\E ( S_{j}^{2}\1_{\{|S_{j}|\ge\varepsilon j\}} )\ =\ o(j^{3-p}) \quad \text{as } j \to \infty,
\end{equation}
and 
\begin{equation} \label{eq:fourth_moment_truncated}
j^{-2} \E (S_{j}^{4}\1_{\{|S_{j}|< \varepsilon j\}})\ =\  
\begin{cases}
\hfill o(j^{3-p}), & p \in (2,3)\\
3 \sigma^4  +  \varepsilon c(j, \varepsilon), & p=3  \\
\end{cases}
\qquad \text{as } j \to \infty,
\end{equation}
where $\limsup_{j \to \infty} |c(j, \varepsilon)|$ is uniformly bounded in $\varepsilon >0$.

To this end, we employ \cite[Claim~22 in Chap.~V,~$\S$5]{Petrov:1975}: for every $p \in (2,3)$, $j \in \N$, and $ x \in \R$,
\begin{equation} \label{eq:Petrovs_bound}
\left|\P\left\{\frac{S_{j}}{\sigma \sqrt{j}}< x\right\}-\P\left\{\mathcal{N}<x\right\}\right|\ \le\ \frac{\psi(\sqrt{j} (1+|x|))}{j^{p/2-1} (1+|x|)^p}, 
\end{equation}
where $\psi$ is a bounded decreasing function on $[1, \infty)$ such that $\psi(x) \to 0 $ as $x \to \infty$. This inequality is also valid for $p=3$ with a constant function $\psi$, see~\cite[Thm.~14 in Chap.~V]{Petrov:1975}.

Notice that for any non-negative random variable $X$ such that $\E X^2<\infty$ and any $a >0$, 
\[
\E ( X^2 \1_{\{ X \ge a \}} ) \ =\ -\int_{[a,\infty)} x^2\  {\rm d} \P\{ X \ge x \}\ =\ a^2 \P\{X \ge a \} + 2 \int_a^\infty x \P\{X \ge x \}\  {\rm d} x. 
\]
Using this formula twice with $X=\frac{|S_{j}|}{\sigma \sqrt{j}}$ and $X= |\mathcal{N}|$, we obtain from \eqref{eq:Petrovs_bound}
\begin{multline}
\bigg | \E  \bigg ( \frac{S_j^2}{\sigma^2 j} \1_{\{ |S_j| \ge \sigma a \sqrt j\}} \bigg)  - \E \big( \mathcal{N}^2  \1_{\{ |\mathcal{N}| \ge  a \}} \big) \bigg|\\
\le \ \frac{2a^2 \psi(\sqrt{j} )}{j^{p/2 -1} a^p}  + 2 \int_a^\infty   \frac{2x \psi(\sqrt{j} )}{j^{p/2 -1} x^p} {\rm d} x\ 
 =\ \frac{2pa^{2-p} \psi(\sqrt{j})}{(p-2) j^{p/2-1}}.\label{eq:petrovs_another_bound}
\end{multline}

This implies \eqref{eq:second_moment_truncated} for $p \in (2,3)$ upon taking $a = \varepsilon \sqrt j /\sigma$ and applying \eqref{eq:normal_truncated_bound} with $\gamma=1$ and $\sigma_j$ replaced by $\sigma$. However, for $p=3$, this gives only $\E ( S_{j}^{2}\1_{\{|S_{j}|\ge\varepsilon j\}} ) =  O(1) $. In order to replace $O(1)$ by $o(1)$, pick $R>\varepsilon$ and write
\[
\E ( S_{j}^{2}\1_{\{|S_{j}|\ge\varepsilon j\}}) \ \le \ R^2 j^2 \P\{|S_{j}|\ge\varepsilon j\} + \E ( S_{j}^{2}\1_{\{|S_{j}|\ge R j\}}).
\]
The first summand is $o(1)$ by a standard estimate of the rate of convergence in the law of large numbers (see~\cite[Theorem 28 in Chap.~IX]{Petrov:1975}). The second summand is bounded by $6 \sigma \psi(1) R^{-1}$  in view of \eqref{eq:petrovs_another_bound} and \eqref{eq:normal_truncated_bound} applied with $a = R \sqrt j/\sigma$. This implies  \eqref{eq:second_moment_truncated} for $p=3$ since $R$ can be arbitrarily large.

Similarly, from \eqref{eq:fourth_moment} and \eqref{eq:Petrovs_bound} it follows that
\[
\bigg | \E  \bigg ( \frac{S_j^4}{\sigma^4 j^2} \1_{\{ |S_j| < \sigma a \sqrt j\}} \bigg)  - \E \big( \mathcal{N}^4  \1_{\{ |\mathcal{N}| <  a \}} \big) \bigg|\ \le\ \frac{2(8-p) a^{4-p} \psi(\sqrt{j})}{(4-p) j^{p/2-1}}.
\]
This implies \eqref{eq:fourth_moment_truncated} for $p \in (2,3]$ upon taking $a = \varepsilon \sqrt j /\sigma$, where for $p=3$, we put
\[
c(j, \varepsilon)\ :=\ \varepsilon^{-1} \bigg [ \E  \bigg ( \frac{S_j^4}{j^2} \1_{\{ |S_j| < \varepsilon j\}} \bigg)  - 3 \sigma^4 \bigg],
\]
which satisfies $|c(j, \varepsilon)|\le  10\sigma^3 \psi(1) + \sigma^4 \varepsilon^{-1} \E \big( \mathcal{N}^4 \1_{\{ |\mathcal{N}| \ge  \varepsilon \sqrt j /\sigma\}} \big)$, hence $\limsup_{j \to \infty} |c(j, \varepsilon)|$ is uniformly bounded in $\varepsilon >0$, as required.

\vspace{.1cm}
We can now prove an upper bound for $\Var(\minlength_{n})$. Using that $g(x) \le \min(x^2/2, x^4/8)$, we have, for every $\varepsilon >0$, 
$$
j\E g\left(\frac{|S_{j}|}{j} \right )  \ \le\ \E \left(\frac{S_{j}^{4}}{8 j^3}\1_{\{|S_{j}|<\varepsilon j\}}\hspace{-2pt}\right)\,+\,\E\left(\frac{S_{j}^{2}}{2j}\1_{\{|S_{j}| \ge\varepsilon j\}}\hspace{-2pt}\right).
$$
Thus, \eqref{eq:second_moment_truncated} and \eqref{eq:fourth_moment_truncated} yield
\[
j\E g\left(\frac{|S_{j}|}{j} \right ) \ \leq \
\begin{cases}
o(j^{2-p}), & p \in (2,3)\\
\frac{1}{8j} \big (3\sigma^4   +   \varepsilon c(j, \varepsilon)+ o(1) \big), & p=3  \\
\end{cases}
\qquad \text{as } j \to \infty.
\]
Upon summation over $j$ and using \eqref{eq:Var=} we derive $\Var(\minlength_{n}) = o(n^{3-p})$ for $p \in (2,3)$, thereby completing the proof of \eqref{eq:prop_variance_poly_asymp}. Similarly, for $p=3$, summing over $j$, dividing by $\log n$, sending $n\to\infty$ and then $\varepsilon \to 0$, we obtain $\limsup_{n\to\infty}\frac{\Var(\minlength_{n})}{\log n} \leq \frac{3\sigma^4}{4}$. 

To establish the matching lower bound for $p=3$ we argue as follows. Since $g(x)\simeq x^4/8$, as $x\to 0$, for every fixed $\varepsilon_1>0$ we can find $\delta_1=\delta_1(\varepsilon_1)>0$ such that
$$
1-\varepsilon_1\leq \frac{8g(x)}{x^4} \quad \text{for }0<x\leq \delta_1.
$$
Therefore,
\begin{align*}
\liminf_{j\to\infty}\E\left(j^2 g\left(\frac{|S_j|}{j}\right)\right)\ &\ge\ \liminf_{j\to\infty}\E\left(j^2 g\left(\frac{|S_j|}{j}\right)\ind_{\{|S_j|\leq \delta_1 j\}}\right)\\
&\ge\ (1-\varepsilon_1)\liminf_{j\to\infty}\E\left(\frac{S_j^4}{8j^2}\ind_{\{|S_j|\leq \delta_1 j\}}\right).
\end{align*}
By \eqref{eq:fourth_moment_truncated} and upon sending $\varepsilon_1$ to zero, this yields
$$
\liminf_{j\to\infty}\E\left(j^2 g\left(\frac{|S_j|}{j}\right)\right)\ \ge\ \frac{3\sigma^4}{8}.
$$
Thus arriving at $\liminf_{n\to\infty}\frac{\Var(\minlength_{n})}{\log n} \geq \frac{3}{4}\sigma^4$, we have completed the proof of \eqref{eq:prop_variance_log_asymp}.

It remains to prove the lower bound \eqref{eq:prop_variance_large} for $\Var \minlength_{n}$. We note that for any $j \in \N$ (see \cite[Lemma~2.1]{Cai+Zhu:2006} or \cite[p.~290]{Erdos:1949}),
\[
\P\{|S_j| \ge j\} \ge  j \rho_j  \P\{|\xi| \ge 2 j\},
\]
where $\rho_j:= \P\{|S_{j-1}| < j\}  - j \P\{|\xi| \ge 2j\}$ and $S_0:=0$. As $g$ is an increasing function, this yields
\[
\E g(|S_j|/j ) \ \ge \ g(1) \P \{|S_j| \ge j\} \ \ge \ g(1) j \rho_j   \P\{|\xi| \ge 2 j\}.
\]
Hence, we obtain by \eqref{eq:Var=}
\[
\Var(\minlength_{n}) = 2 \sum_{j=1}^{n} j \E g(|S_j|/j ) + O(1) \ge 2g(1) \sum_{j=1}^{n}  j^2 \rho_j \P \{|\xi| \ge 2j\} + O(1)
\] 
as $n \to \infty$. This gives the desired bound $\Var(\minlength_{n}) \ge 0.02 n^3  \P \{|\xi| \ge 2n\} + O(1)$ since $2g(1)/3>0.02$ and $\lim_{j \to \infty} p_j = 1$. The last statement holds by the assumptions $\E \xi^2<\infty$ and $\E \xi=0$. This finishes the proof of Theorem \ref{thm:moments}.
\end{proof}

\subsection{Proof of Theorem \ref{thm:infinite_variance}}

We start by showing that the random series representing the limit random variable in Theorem \ref{thm:infinite_variance} is finite a.s.

\begin{proposition}\label{prop:limit_is_finite}
Let $(Z_{1},Z_{2},\ldots)$ be a random sequence having a Poisson--Dirichlet distribution with parameter $\theta=1$. Further, let $\mathcal{S}_{\alpha}^{(k)}=(\mathcal{S}_{\alpha}^{(k)}(t))_{t\in[0,1]}$, $k=1,2,\ldots,$ be independent copies of the strictly stable process $\mathcal{S}_{\alpha}$ with index $\alpha\in(1,2)$. Then,
$$
\sum_{k=1}^{\infty}\frac{(\mathcal{S}_{\alpha}^{(k)}(Z_{k}))^{2}}{Z_{k}}<\infty\quad\text{a.s.}
$$
\end{proposition}
\begin{proof}
Fix an arbitrary $\delta\in (0,\alpha/2)$. Using that $x\mapsto x^{\delta}$ is subadditive and then the self-similarity of the process $\mathcal{S}_{\alpha}$, we obtain
\[
\E\left(\sum_{k=1}^{\infty}\frac{(\mathcal{S}_{\alpha}^{(k)}(Z_{k}))^{2}}{Z_{k}}\right)^{\delta}\ \le \ \sum_{k=1}^{\infty}\E\left(\frac{(\mathcal{S}_{\alpha}^{(k)}(Z_{k}))^{2}}{Z_{k}}\right)^{\delta} \ = \ \E|\mathcal{S}_{\alpha}(1)|^{2\delta}\sum_{k=1}^{\infty}\E Z_{k}^{(2/\alpha-1)\delta}.
\]
Since $\E |\mathcal{S}_{\alpha}(1)|^{2\delta}<\infty$ (as $2\delta<\alpha$), it remains to check that $\sum_{k=1}^{\infty}\E Z_{k}^{(2/\alpha-1)\delta}$ is finite. To this end, formula (2.1) from \cite{Arratia+Barbour+Tavare:2006} with $\phi(x)=x^{(2/\alpha-1)\delta}$ can be used to see that
$$
\sum_{k=1}^{\infty}\E Z_{k}^{(2/\alpha-1)\delta}\ =\ \int_{0}^{1} \frac{\phi(x)}{x}{\rm d}x\ <\ \infty.
$$
This completes the proof and we note that the same argument applies to any strictly stable process of index $\alpha\in (0,2)$.
\end{proof}

Before passing to the proof of Theorem \ref{thm:infinite_variance}, we give an auxiliary lemma.

\begin{lemma}\label{lem:lipschitz}
If $\E |\xi|<\infty$, then
$$
\E \left|\left(\sqrt{i^{2}+S_i^{2}}-i\right)-\left(\sqrt{j^{2}+S_{j}^{2}}-j\right)\right|\ \le\ (2+\E|\xi|)\,|i-j|
$$
holds for any $i,j\in\N$.
\end{lemma}

\begin{proof}
It is obviously enough to show that
$$
\E \left|\sqrt{i^{2}+S_i^{2}}-\sqrt{j^{2}+S_{j}^{2}}\right|\ \le\ (1+\E|\xi|)|i-j|,
$$
which follows from
\begin{align*}
\E&\left|\sqrt{i^{2}+S_i^{2}}-\sqrt{j^{2}+S_{j}^{2}}\right|\ =\ \E \left|\frac{i^{2}+S_i^{2}-j^{2}-S_{j}^{2}}{\sqrt{i^{2}+S_i^{2}}+\sqrt{j^{2}+S_{j}^{2}}}\right|\\
&\le\ \E \frac{|i^{2}-j^{2}|}{\sqrt{i^{2}+S_i^{2}}+\sqrt{j^{2}+S_{j}^{2}}}\,+\,
\E \frac{|S_i^{2}-S_{j}^{2}|}{\sqrt{i^{2}+S_i^{2}}+\sqrt{j^{2}+S_{j}^{2}}}\\
&\le\ \frac{|i^{2}-j^{2}|}{\sqrt{i^{2}}+\sqrt{j^{2}}}\,+\,
\E\frac{|S_i^{2}-S_{j}^{2}|}{\sqrt{S_i^{2}}+\sqrt{S_{j}^{2}}}\\
&=\ |i-j|+\E ||S_i|-|S_{j}||\ \le\ |i-j|+\E |S_i-S_{j}|\\
&=\ |i-j|+\E |S_{i-j}|\ \le\ (1+\E|\xi|)|i-j|
\end{align*}
for any $i,j\in\N$.
\end{proof}

\begin{proof}[Proof of Theorem \ref{thm:infinite_variance}]
Fixing a coupling between $(Z_{n,k})$ and $(Z_{k})$ such that \eqref{eq:L_{1}_conv_{j}oint_large_cycles} holds, we show first that
\begin{equation}
\label{eq:infinite_mean_approximation1}
\frac{1}{b_{n}}\Bigg|\sum_{k=1}^{n}\left(\sqrt{Z_{n,k}^{2}+S^{2}_{k,Z_{n,k}}}-Z_{n,k}\right)\ - \ \sum_{k=1}^{n}\left( \sqrt{\lfloor nZ_{k}\rfloor^{2}+S^{2}_{k,\lfloor nZ_{k}\rfloor}}-\lfloor nZ_{k}\rfloor\right)\Bigg|\ \toprobab\ 0.
\end{equation}
Recall that $b_{n}=a_{n}^{2}/n$ with $a_{n}$ given by \eqref{eq:clt_for_rw}. By Lemma \ref{lem:lipschitz} and use of the triangle inequality, the expectation of the difference in \eqref{eq:infinite_mean_approximation1} is bounded by a constant times
$$
\frac{n}{a_{n}^{2}}\sum_{k=1}^{n}\E |Z_{n,k}-\lfloor nZ_{k}\rfloor |\ \le\ \frac{n^{2}}{a_{n}^{2}}\sum_{k=1}^{\infty}\E \left|\frac{Z_{n,k}}{n}-Z_{k} \right|\,+\,\frac{n}{a_{n}^{2}}\,\E \sum_{k=1}^{\infty}\{nZ_{k}\},
$$
where $\{x\}$ denotes the fractional part of $x\in\R$. Then use \eqref{eq:L_{1}_conv_{j}oint_large_cycles} to assess that the first summand on the right is of the same order as $a_{n}^{-2}n\log n/4$ and thus tending to zero as $n\to\infty$ because $(a_{n}^{2})_{n\in\N}$ is regularly varying with index $2/\alpha>1$. The second summand tends to zero by the same reasoning and the fact that
\begin{equation*} 
\E \sum_{k=1}^{\infty}\{nZ_{k}\}\ \simeq\ \frac{1}{2}\log n\quad\text{as }n\to\infty,
\end{equation*}
which can be justifed as follows: by formula (2.1) in \cite{Arratia+Barbour+Tavare:2006} with $\phi(x)=\{nx\}$, we have
\begin{gather*}
\E \sum_{k=1}^{\infty}\{nZ_{k}\}\ =\ \int_{0}^{1} \frac{\{nx\}{\rm d}x}{x}\ =\ \int_{0}^{n} \frac{\{y\}{\rm d}y}{y}\ =\ \sum_{k=0}^{n-1}\int_{k}^{k+1}\frac{y-k}{y}{\rm d}y\\
=\ 1+\sum_{k=1}^{n-1}\left(1-k\log \left(\frac{k+1}{k}\right)\right)\ \simeq\ 1+\frac{1}{2}\sum_{k=1}^{n-1}\frac{1}{k}\ \simeq\ \frac{1}{2}\log n
\end{gather*}
as $n\to\infty$. Now \eqref{eq:infinite_mean_approximation1} follows by these estimates and Markov's inequality.

\vspace{.1cm}
In view of \eqref{eq:infinite_mean_approximation1} and representation \eqref{eq:basic_representation1}, it remains to prove
\begin{equation}
\frac{\sum_{k=1}^{n}\left(\sqrt{\lfloor nZ_{k}\rfloor^{2}+S^{2}_{k,\lfloor nZ_{k}\rfloor}}-\lfloor nZ_{k}\rfloor\right)}{b_{n}}\ \todistr\ \sum_{k=1}^{\infty}\frac{(\mathcal{S}_{\alpha}^{(k)}(Z_{k}))^{2}}{2Z_{k}}.
\end{equation}
To this end, we use Theorem 3.2 in \cite{Billingsley:1999} for which the following two assertions must be verified: First, for any fixed $m\in\N$,
\begin{equation}\label{eq:infinite_variance_bill1}
\frac{\sum_{k=1}^{m}\left(\sqrt{\lfloor nZ_{k}\rfloor^{2}+S^{2}_{k,\lfloor nZ_{k}\rfloor}}-\lfloor nZ_{k}\rfloor\right)}{b_{n}}\ \todistr\ \sum_{k=1}^{m}\frac{(\mathcal{S}_{\alpha}^{(k)}(Z_{k}))^{2}}{2Z_{k}},
\end{equation}
and, second, for any $\varepsilon>0$,
\begin{equation}\label{eq:infinite_variance_bill2}
\lim_{m\to\infty}\limsup_{n\to\infty}\,\P\left\{\left|\sum_{k=m}^{n}\left(\sqrt{\lfloor nZ_{k}\rfloor^{2}+S^{2}_{k,\lfloor nZ_{k}\rfloor}}-\lfloor nZ_{k}\rfloor\right)\right|>\varepsilon b_{n}\right\}\ =\ 0.
\end{equation}

Note that \eqref{eq:infinite_variance_bill1} is obviously equivalent to
$$
\sum_{k=1}^{m}\frac{S^{2}_{k,\lfloor nZ_{k}\rfloor}}{a_{n}^{2}\left(\sqrt{\lfloor nZ_{k}\rfloor^{2}+S^{2}_{k,\lfloor nZ_{k}\rfloor}}+\lfloor nZ_{k}\rfloor\right)/n}\ \todistr\ \sum_{k=1}^{m}\frac{(\mathcal{S}_{\alpha}^{(k)}(Z_{k}))^{2}}{2Z_{k}}.
$$
But this is true by the continuous mapping theorem and the joint convergence
\[
\left(\frac{S^{2}_{k,\lfloor nZ_{k}\rfloor}}{a_{n}^{2}},\frac{\sqrt{\lfloor nZ_{k}\rfloor^{2}+S^{2}_{k,\lfloor nZ_{k}\rfloor}}+\lfloor nZ_{k}\rfloor}{n}\right)_{1\le k\le m} \todistr\ ((\mathcal{S}_{\alpha}^{(k)}(Z_{k}))^{2},2Z_{k})_{1\le k\le m},
\]
which in turn holds because $a_{n}=o(n)$ as $n\to\infty$ and  $(S_{i,j})_{i, j \in \N}$ and $(Z_{k})_{k \in \N}$ are independent.

\vspace{.1cm}
For \eqref{eq:infinite_variance_bill2}, we argue as follows. Using formula \eqref{eq:lem_ui_proof00}, it is enough to check that
\begin{equation}\label{eq:infinite_variance_bill21}
\lim_{m\to\infty}\limsup_{n\to\infty}\,\P\left\{\sum_{k=m}^{n} \frac{S^{2}_{k,\lfloor nZ_{k}\rfloor}}{2 \lfloor nZ_{k}\rfloor}\ind_{\{nZ_{k}\ge 1\}}>\varepsilon b_{n}\right\}\ =\ 0.
\end{equation}
Fix $\delta\in(0,\alpha/2)$. Then, using the subadditivity of $x\mapsto x^{\delta}$ and Markov's inequality, we see that \eqref{eq:infinite_variance_bill21} is a consequence of
\begin{equation}\label{eq:infinite_variance_bill22}
\lim_{m\to\infty}\limsup_{n\to\infty} \frac{1}{b_{n}^\delta} \sum_{k=m}^{n}\E \left(\frac{S^{2}_{k,\lfloor nZ_{k}\rfloor}}{2 \lfloor nZ_{k}\rfloor}\ind_{\{nZ_{k}\ge 1\}}\right)^{\delta}=0.
\end{equation}
By Lemma 5.2.2 in \cite{Ibragimov+Linnik:1971}, there exists a constant $C=C_{\delta,\alpha}>0$ such that
\begin{equation} \label{eq:moments_bound}
\E |S_l |^{2\delta}\ \le\ Ca^{2\delta}_l,\quad l\in\N.
\end{equation}
Therefore,
\begin{gather*}
\E\left(\frac{S^{2}_{k,\lfloor nZ_{k}\rfloor}}{2 \lfloor nZ_{k}\rfloor}\ind_{\{nZ_{k}\ge 1\}}\right)^{\delta}\ =\ \sum_{l=1}^{n}\P\{\lfloor nZ_{k}\rfloor=l\}\,\E\left(\frac{S^{2}_{l}}{2l}\right)^{\delta}\\
\le\ C2^{-\delta}\sum_{l=1}^{n}\P\{\lfloor nZ_{k}\rfloor=l\}\frac{a^{2\delta}_l}{l^{\delta}}\ =\ C2^{-\delta}\,\E b^{\delta}_{\lfloor nZ_{k}\rfloor}\ind_{\{\lfloor nZ_{k}\rfloor\ge 1\}},
\end{gather*}
and \eqref{eq:infinite_variance_bill22} will follow once having shown that
\begin{equation}\label{eq:infinite_variance_bill23}
\lim_{m\to\infty}\limsup_{n\to\infty} \frac{1}{b_{n}^\delta} \sum_{k=m}^{n} \E \left(b^{\delta}_{ \lfloor nZ_{k}\rfloor}\ind_{\{ \lfloor nZ_{k}\rfloor\ge 1\}}\right)=0.
\end{equation}

Fix an arbitrary $\delta'\in(0,(\frac{2}{\alpha}-1)\delta)$. Since $(b_{n}^{\delta})_{n\in\N}$ is regularly varying with index $(\frac{2}{\alpha}-1)\delta$, we can apply Potter's bound to the slowly varying sequence $(b_{n}^{\delta}/ n^{(2/\alpha-1)\delta})_{n\in\N}$, see \cite[Thm.~1.5.6(ii)]{BGT}. In combination with $ \lfloor nZ_{k}\rfloor\le n$, this gives
$$
\frac{b^{\delta}_{ \lfloor nZ_{k}\rfloor}}{b_{n}^{\delta}}\ \le\ {\rm const}\cdot \left(\frac{\lfloor nZ_{k}\rfloor}{n}\right)^{(2/\alpha-1)\delta-\delta'}\ \le\ {\rm const}\cdot Z_{k}^{(2/\alpha-1)\delta-\delta'}.
$$
Moreover, by formula (2.1) in \cite{Arratia+Barbour+Tavare:2006},
\begin{equation}\label{eq:infinite_variance_final_formula}
\sum_{k=1}^{\infty}\E Z_{k}^{(2/\alpha-1)\delta-\delta'}\ =\ \int_{0}^{1}x^{(2/\alpha-1)\delta-\delta'-1}{\rm d}x<\infty.
\end{equation}
Hence \eqref{eq:infinite_variance_bill23} follows and the proof of \eqref{eq:infinite_variance_bill2} is complete.
\end{proof}

\subsection{Proof of Theorem \ref{thm:infinite_mean}}\label{sec:proof_inf_mean}

Put
$$
\tau_{n}\,:=\,\argmax_{0\le k\le n}S_{k}\quad\text{and}\quad\kappa_{n}\,:=\,\argmin_{0\le k\le n}S_{k},
$$
(where, for the sake of definiteness, the position of the \emph{first} maximum or minimum is taken) and further,
$$
M_{n}=\max_{0\le k\le n}S_{k}\quad\text{and}\quad m_{n}=\min_{0\le k\le n}S_{k},
$$
see Figure~\ref{fig:proof_inf_mean}.

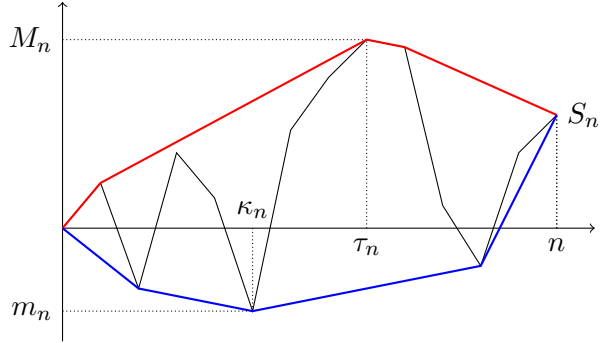
\begin{figure}[!hbtp]
\centering
\begin{tikzpicture}
\draw[->] (0,-1.5) -- (0,3.0);
\draw[->] (0,0) -- (7.0,0.0);
\draw[-] (0,0) -- (0.5,0.6) -- (1.0,-0.8) -- (1.5,1.0) -- (2.0,0.4) -- (2.5,-1.1) -- (3.0, 1.3) -- (3.5, 2.0) -- (4.0,2.5) -- (4.5,2.4) -- (5.0,0.3) -- (5.5,-0.5) -- (6.0,1.0) -- (6.5,1.5);
\draw[densely dotted] (4.0,0) -- (4.0,2.5);
\draw[densely dotted] (0,2.5) -- (4.0,2.5);
\draw[densely dotted] (6.5,0) -- (6.5,1.5);
\draw[densely dotted] (2.5,0) -- (2.5,-1.1);
\draw[densely dotted] (6.5,0) -- (6.5,1.5);
\draw[densely dotted] (0,-1.1) -- (2.5,-1.1);
\node[left,scale=1.0] (c) at (0,2.5) {$M_{n}$};
\node[left,scale=1.0] (c) at (0,-1.1) {$m_{n}$};
\node[above,scale=1.0] (c) at (2.5,0) {$\kappa_{n}$};
\node[below,scale=1.0] (c) at (4.0,0.0) {$\tau_{n}$};
\node[right,scale=1.0] (c) at (6.5,1.5) {$S_{n}$};
\node[below,scale=1.0] (c) at (6.5,0.0) {$n$};
\draw[thick,red] (0,0) -- (0.5,0.6) -- (4.0,2.5) -- (4.5,2.4) -- (6.5,1.5);
\draw[thick,blue] (0,0) -- (1.0,-0.8) -- (2.5,-1.1) -- (5.5,-0.5) -- (6.5,1.5);
\end{tikzpicture}
\caption{The proof of Theorem \ref{thm:infinite_mean}. The thin black piecewise linear line is a linearly interpolated random walk $(S_{k})_{k=0,\ldots,n}$, the thick red line consisting of four line segments is its concave majorant,
the thick blue line is its convex minorant.}
\label{fig:proof_inf_mean}
\end{figure}

Our reasoning is geometric and based on  a simple comparison argument. First of all note that the concave majorant consists of two piecewise linear subparts: one connecting the origin and the point $(\tau_{n},M_{n})$, and the other connecting $(\tau_{n},M_{n})$ and $(n,S_{n})$. Denote their lengths by $\majlength_{1,n}$ and $\majlength_{2,n}$, respectively, thus $\majlength_{n}=\majlength_{1,n}+\majlength_{2,n}$. The triangle inequality applied to every segment of the concave majorant provides
\begin{gather*}
M_{n}\,\le\,\majlength_{1,n}\,\le\,M_{n}+\tau_{n}\,\le\,M_{n}+n,
\shortintertext{and also}
M_{n}-S_{n}\,\le\,\majlength_{2,n}\,\le\,M_{n}-S_{n} + n-\tau_{n}\,\le\,M_{n}-S_{n} + n.
\end{gather*}
Consequently,
$$ 2M_{n}-S_{n}\,\le\,\majlength_{n}\,\le\,2M_{n}-S_{n}+2n. $$
Similarly,
$$
S_{n}-2m_{n}\,\le\,\minlength_{n}\,\le\,S_{n}-2m_{n} + 2n.
$$
To complete he proof of Theorem \ref{thm:infinite_mean}, it remains to note that $n/a_{n}\to 0$ because $(a_{n})_{n\in\N}$ is regularly varying with index $1/\alpha>1$ and
$$
\left(\frac{S_{n}}{a_{n}},\frac{M_{n}}{a_{n}},\frac{m_{n}}{a_{n}}\right)\ \todistr\ \big (\mathcal{S}_{\alpha}(1),\sup_{t\in[0,1]}\mathcal{S}_{\alpha}(t),\inf_{t\in[0,1]}\mathcal{S}_{\alpha}(t) \big)
$$
as an immediate consequence of \eqref{eq:clt_for_rw} and the continuous mapping theorem.

\section{The case of nonzero mean}\label{sec:non_zero_mu}

Throughout this section, we assume that $\mu:=\E\xi$ exists and is nonzero which rules out case (C). We are left with two possibilities:
\begin{itemize}
\item[(A$'$)] $\E \xi^{2}<\infty$ and $\mu\neq 0$;
\item[(B$'$)] $\E \xi^{2}=\infty$, $\mu\neq 0$, and the law of $\xi$ lies in the domain of attraction of an $\alpha$-stable law with $\alpha\in(1,2]$.
\end{itemize}
We stress that case (B$'$) includes $\alpha=2$, as opposed to case (B).

\vspace{.1cm}
The essence of the next theorem is that in both cases (A$'$) and (B$'$), the distributional behavior of $\minlength_{n}$ as $n\to\infty$ coincides with that of $S_{n}$ up to a linear centering and a scaling as in Theorem~1.7 in~\cite{McRedmond+Wade:2018}. We will see that the same holds for $L_{n}$, the perimeter of the convex hull of $\{(j,S_{j}): j=0,\ldots,n\}$.

\begin{theorem}\label{thm:non_zero_mu}
In the cases $(A')$ and $(B')$, we have
\begin{equation} \label{eq:conv_min_CLT}
\frac{\minlength_{n}-\sqrt{1+\mu^{2}}\,n}{a_{n}}\ \todistr\ \frac{\mu}{\sqrt{1+\mu^{2}}}\,\mathcal{S}_{\alpha}(1),
\end{equation}
where $a_{n} := \sqrt{n}$, $\sigma^{2}:=\E\xi^{2} -\mu^{2}$, and $\mathcal{S}_{2}(1):=\mathcal{N}(0,\sigma^{2})$ in case $(A')$, and
$(a_{n})_{n \in \N}$ is a sequence and $\mathcal{S}_{\alpha}(1)$ an $\alpha$-stable random variable such that $(S_{n}-\mu n)/a_{n}\stackrel{d}{\to}\mathcal{S}_{\alpha}(1)$.

\vspace{.1cm}\noindent
The convergence \eqref{eq:conv_min_CLT} does also hold with $\majlength_{n}$ and $L_{n}/2$ in the place of $\minlength_{n}$.
\end{theorem}

The result for $L_{n}$ in case (A$'$) is a particular case of Theorem~1.2 in~\cite{Wade+Xu:2015b} (combined with Theorem~1.8 in~\cite{McRedmond+Wade:2018}), which applies because $(n, S_{n})_{n \in \N_{0}}$ can be regarded as a random walk in the plane. The approach of~\cite{McRedmond+Wade:2018, Wade+Xu:2015b} is different from the one employed here: we combine the formula $L_{n}=\majlength_{n} + \minlength_{n}$ with the limit result~\eqref{eq:conv_min_CLT} (or, to be more precise, \eqref{eq:non_zero_mean_proof1}) for $\minlength_{n}$ and its version for $\majlength_{n}$, both obtained from representation~\eqref{eq:basic_representation2}. Recall that this argument does not work when $\mu =0$; cf.~Remark~\ref{rem:non-joint}.

\begin{proof}
Having $(S_{n}-\mu n)/a_{n}\stackrel{d}{\to}\mathcal{S}_{\alpha}(1)$ in both cases (A$'$) and (B$'$) (naturally with $\alpha=2$ in (A$'$)), \eqref{eq:conv_min_CLT} follows if we can show that
\begin{equation}\label{eq:non_zero_mean_proof1}
\frac{1}{a_{n}}\left|\minlength_{n}-\sqrt{1+\mu^{2}}n - \frac{\mu}{\sqrt{1+\mu^{2}}}(S_{n} - \mu n)\right|\ \toprobab\ 0,
\end{equation}
which in case (A$'$) is just a particular case of Theorem~1.7 in~\cite{McRedmond+Wade:2018}.

\vspace{.1cm}
We use an extended version of representation \eqref{eq:basic_representation2}, viz.
\begin{equation}\label{eq:basic_representation2'}
(\minlength_{n}-n, S_{n})\ \od\ \Bigg( \sum_{j=1}^{n}\sum_{i=1}^{K_{n,j}}\sqrt{j^{2}+S^{2}_{i,j}}-j, S_{n} \Bigg).
\end{equation}
Defining
$$
\widetilde{S}_{i,j}:=S_{i,j}-\mu j,\quad
$$
for $i,j\in\N$, we have
$$
\sum_{j=1}^{n}\sum_{i=1}^{K_{n,j}}\widetilde{S}_{i,j}\ =\ S_{n}-\mu n,
$$
and
\begin{align*}
\Delta_{n}\ :=\ &\minlength_{n}-\sqrt{1+\mu^{2}}\,n - \frac{\mu}{\sqrt{1+\mu^{2}}}(S_{n} - \mu n)\\
\od\ &\sum_{j=1}^{n}\sum_{i=1}^{K_{n,j}}\left(\sqrt{j^{2}+(\mu j+\widetilde{S}_{i,j})^{2}}-\sqrt{1+\mu^{2}}\,j-\frac{\mu}{\sqrt{1+\mu^{2}}}\widetilde{S}_{i,j}\right)\\
=\ &\sqrt{1+\mu^{2}}\sum_{j=1}^{n}\sum_{i=1}^{K_{n,j}}j\left(\sqrt{1+\frac{\widetilde{S}_{i,j}^{2}+2\mu j \widetilde{S}_{i,j}}{(1+\mu^{2})j^{2}}}-1-\frac{\mu}{1+\mu^{2}}\frac{\widetilde{S}_{i,j}}{j}\right).
\end{align*}
Further, we will use the inequalities
\begin{gather}\label{eq:non_zero_mean_proof2}
1+\frac{y}{2}\le \sqrt{1+x+y}\quad\text{for}\quad x\ge y^{2}/4,\ x+y\ge -1,
\shortintertext{and}
\label{eq:non_zero_mean_proof3}
\sqrt{1+z}\le 1+\frac{z}{2}\quad\text{for}\quad z\ge -1.
\end{gather}
The first one with $x:=\widetilde{S}_{i,j}^{2}/((1+\mu^{2})j^{2})$ and $y:=2\mu j \widetilde{S}_{i,j}/((1+\mu^{2})j^{2})$ yields $\Delta_{n}\ge 0$ a.s. Then the second one with $z:=\widetilde{S}_{i,j}^{2}+2\mu j\widetilde{S}_{i,j}/((1+\mu^{2})j^{2})$ implies
$$
\P\{|\Delta_{n}|>\varepsilon a_{n}\}\ \le\
\P\left\{\sum_{j=1}^{n}\sum_{i=1}^{K_{n,j}}\frac{\widetilde{S}^{2}_{i,j}}{2\sqrt{1+\mu^{2}}j}>\varepsilon a_{n}\right\}
$$
for every $\varepsilon>0$.

\vspace{.1cm}
Fix $\delta\in (0,\alpha/2)$. Then, by subadditivity of $x\mapsto x^{\delta}$ and Markov's inequality, it suffices to check that
\begin{equation} \label{eq:from_Markov}
\lim_{n\to\infty}\frac{1}{a_{n}^{\delta}} \E\left( \sum_{j=1}^{n}\sum_{i=1}^{K_{n,j}} \frac{\widetilde{S}^{2 \delta}_{i,j}}{j^\delta}\right) =0.
\end{equation}
We use \eqref{eq:moments_bound}, the independence of $\widetilde{S}_{i,j}$ and~$K_{n,j}$ and the formula $\E K_{n,j}=j^{-1}$ stated in \eqref{eq:K_{n}j_moments} to infer
$$
\E\left( \sum_{j=1}^{n}\sum_{i=1}^{K_{n,j}} \frac{\widetilde{S}^{2 \delta}_{i,j}}{j^\delta}\right)\ =\ \sum_{j=1}^{n}\frac{1}{j}\frac{\E S^{2\delta}_{j}}{j^{\delta}}\ \le\ {\rm const}\cdot\sum_{j=1}^{n}\frac{a_{j}^{2\delta}}{j^{1+\delta}},
$$
where the last inequality is a consequence of \eqref{eq:moments_bound}. Since $(a_{j})_{j\in\N}$ is regularly varying with index $1/\alpha$, the sequence $(j^{-1-\delta}a_{j}^{2\delta})_{j\in\N}$ is regularly varying with  index $2\delta/\alpha-1-\delta\ge -1$ and the sum is regularly varying with index $2\delta/\alpha-\delta$. And since the latter is smaller than $\delta/\alpha$, the index of regular variation of $(a_{n}^{\delta})_{j\in\N}$ for $\alpha>1$, we conclude \eqref{eq:from_Markov} thus completing the proof of \eqref{eq:conv_min_CLT}.

\vspace{.1cm}
Since the representation \eqref{eq:basic_representation2'} is also valid for $\majlength_{n}$ instead of $\minlength_{n}$, the convergence \eqref{eq:conv_min_CLT} and the coupling \eqref{eq:non_zero_mean_proof1} (between $\minlength_{n}$ and a normalization of~$S_{n}$) remain true for $\majlength_{n}$ in the place of $\minlength_{n}$. By combining  \eqref{eq:non_zero_mean_proof1} for both $\minlength_{n},\,\majlength_{n}$ finally provides  \eqref{eq:conv_min_CLT} for $L_{n}/2=(\minlength_{n} + \majlength_{n})/2$.
\end{proof}

\section*{Acknowledgment}

G.~Alsmeyer and Z.~Kabluchko were partially funded by the Deutsche Forschungsgemeinschaft (DFG) under Germany's Excellence Strategy EXC 2044--390685587, Mathematics M\"unster: Dynamics--Geometry--Structure. V.~Vysotsky was supported in part by the RFBR Grant 19-01-00356.

\vspace{1cm}

\footnotesize

\textsc{Gerold Alsmeyer and Zakhar Kabluchko:}  Institute of Mathematical Stochastics,
Department of Mathematics and Computer Science, University of M\"unster, Orl\'eans-Ring 10,
D-48149 M\"unster, Germany\\
\textit{E-mails}: \texttt{gerolda@uni-muenster.de, zakhar.kabluchko@uni-muenster.de}

\bigskip

\textsc{Alexander Marynych:} Faculty of Computer Science and Cybernetics, Taras Shev\-chen\-ko National University of Kyiv,
01601 Kyiv, Ukraine\\
\textit{E-mail}: \texttt{marynych@unicyb.kiev.ua}

\bigskip

\textsc{Vladislav Vysotsky:} University of Sussex, Sussex House, Falmer Brighton, BN1 9RH United Kingdom and St.\ Petersburg Department of Steklov Mathematical Institute, Fontanka~27,
191011 St.\ Petersburg, Russia\\
\textit{E-mail}: \texttt{V.Vysotskiy@sussex.ac.uk}

\end{document}